\documentclass[11pt]{amsart}
\usepackage[colorlinks=true, pdfstartview=FitV, linkcolor=blue, 
citecolor=blue]{hyperref}

\usepackage{amsmath, amssymb, amscd}
\usepackage{a4wide}
\usepackage{graphicx}

\newtheorem{thm}{Theorem}
\newtheorem{fact}{Fact}
\newtheorem{lem}{Lemma}
\newtheorem{coro}{Corollary}
\newtheorem{prop}{Proposition}

\theoremstyle{definition}

\newtheorem{rem}{Remark}
\newtheorem{example}{Example}

\newcommand{\ts}{\hspace{0.5pt}}
\newcommand{\nts}{\hspace{-0.5pt}}

\newcommand{\AAA}{\mathbb{A}}
\newcommand{\RR}{\mathbb{R}\ts}

\newcommand{\ZZ}{\mathbb{Z}}
\newcommand{\TT}{\mathbb{T}}
\newcommand{\NN}{\mathbb{N}}
\newcommand{\XX}{\mathbb{X}\ts}
\newcommand{\YY}{\mathbb{Y}}

\newcommand{\cA}{\mathcal{A}}
\newcommand{\cB}{\mathcal{B}}
\newcommand{\cG}{\mathcal{G}}
\newcommand{\cH}{\mathcal{H}}
\newcommand{\cS}{\mathcal{S}}
\newcommand{\cR}{\mathcal{R}}
\newcommand{\cL}{\mathcal{L}}
\newcommand{\cO}{\mathcal{O}}

\newcommand{\Aut}{\mathrm{Aut}}
\newcommand{\card}{\mathrm{card}}
\newcommand{\ord}{\mathrm{ord}}
\newcommand{\dX}{\mathrm{d}^{}_{\XX}}
\newcommand{\id}{\mathrm{Id}}
\newcommand{\nix}{\, \underline{\! 0 \! } \,}

\newcommand{\lcm}{\mathrm{lcm}}
\newcommand{\GL}{\mathrm{GL}}

\newcommand{\exend}{\hfill $\Diamond$}
\newcommand{\bs}{\boldsymbol}

\usepackage{color}

\begin{document}

\title{Reversing and extended symmetries of shift spaces}

\author{M.~Baake}
\address{Faculty of Mathematics, Universit\"{a}t Bielefeld, 
  Box 100131, 33501 Bielefeld, Germany}
\email{mbaake@math.uni-bielefeld.de}

\author{J.A.G.~Roberts}
\address{School of Mathematics and Statistics, UNSW, Sydney, 
  NSW 2052, Australia}
\email{jag.roberts@unsw.edu.au}

\author{R.~Yassawi}
\address{IRIF, Universit\'{e} Paris-Diderot --- Paris 7,
Case 7014, 75205 Paris Cedex 13, France}
\email{ryassawi@gmail.com}

\begin{abstract}
  The reversing symmetry group is considered in the setting of
  symbolic dynamics. While this group is generally too big to be
  analysed in detail, there are interesting cases with some form of
  rigidity where one can determine all symmetries and reversing
  symmetries explicitly. They include Sturmian shifts as well as
  classic examples such as the Thue--Morse system with various
  generalisations or the Rudin--Shapiro system.  We also look at
  generalisations of the reversing symmetry group to
  higher-dimensional shift spaces, then called the group of
  \emph{extended symmetries}. We develop their basic theory for
  faithful $\ZZ^{d}$-actions, and determine the extended symmetry
  group of the chair tiling shift, which can be described as a model
  set, and of Ledrappier's shift, which is an example of algebraic
  origin.
\end{abstract}

\maketitle

\section{Introduction}

The symmetries of a (topological) dynamical system always have been an
important tool to analyse the structure of the system; see
\cite{Gol,Sev,LR,OS} and references therein for background.  In
particular, we mention the structure of periodic or closed orbits or,
more generally, the problem to decide on topological conjugacy of two
dynamical systems.  For (concrete) dynamical systems defined by the
action of a homeomorphism $T$ on some (compact) topological space
$\XX$, it had been realised early on that a \emph{time reversal
  symmetry} has powerful consequences on the structure of the dynamics
as well; see \cite{Sev} and references therein.  This development was
motivated by the importance of time reversal in fundamental equations
of physics.

Much later, this approach was re-analysed and extended in a more
group-theoretic setting \cite{Lamb,G1}, via the introduction of the
(topological) symmetry group $\cS (\XX)$ {of $(\XX,T)$} and its
extension to the reversing symmetry group $\cR (\XX)$.  The latter,
which is the group of all self-conjugacies and \emph{flip} conjugacies
of $(\XX,T)$, need not be mediated by an involution as in the
original setting; see \cite{BR-nonlin,BR-poly,LPS} for examples and
\cite{LR,BR-Bulletin,OS} for some general and systematic results.
Previous studies have often been restricted to concrete systems such
as trace maps \cite{RB}, toral automorphisms \cite{BR-nonlin} or
polynomial automorphisms of the plane \cite{BR-poly}; see \cite{OS}
for a comprehensive overview from an algebraic perspective.

It seems natural to also consider invertible \emph{shift spaces} from
symbolic dynamics \cite{Hed}, where shifts of finite type (also known
as topological Markov shifts) have been studied quite extensively; see
\cite{LPS,KLP} and references given there as well as
\cite[Ch.~3]{Kit-Book}. However, the symmetry group of a shift space can
be huge, with arbitrary finite groups and even free groups appearing
as subgroups; compare \cite[Thm.~3.3.18]{Kit-Book} and
\cite[p.~437]{LM}.  In particular, they need not be amenable, which
turns a classification into a wild problem.  Nevertheless, the
existence of reversing symmetries can be (and has been) studied, with
interesting consequences for the isomorphism problem \cite{GLR} or the
nature of dynamical zeta functions \cite{KLP}. One example for the
relevance of the reversing symmetry group of a minimal shift $(\XX,S)$
is that it yields the automorphism group of the \emph{topological full
  group} of $(\XX,S)$; compare \cite[Thm.~4.2 and
Cor.~4.4]{GPS}. These topological full groups have recently been shown
to be finitely generated, simple, infinite and amenable \cite{JM}.

The situation changes significantly in the presence of algebraic,
geometric or topological constraints that lead to symmetry rigidity,
as present in Sturmian sequences \cite{Olli} or certain ergodic
substitution dynamical systems \cite{Coven,CQY}, where the symmetry
group is minimal in the sense that it just consists of the group
generated by the shift itself. More generally, we speak of
\emph{symmetry rigidity} when the symmetry group coincides with the
group generated by the shift or is just a finite-index extension of
it. The relevance of this scenario was recently realised and
investigated by several groups \cite{CK0,CK,DDMP,CQY}. This led to
interesting deterministic cases (hence with zero topological entropy),
where one can now also determine $\cR (\XX)$. As we shall see,
symmetry rigidity is neither restricted to minimal shifts, nor to
shifts with zero entropy. In general, the shift space $\XX$ (with its
Cantor structure) is important and complicated, while the shift map
itself is simple --- in contrast to many previously studied cases of
reversibility, where the space is simple, while the mapping requires
more attention. Typically, the latter situation leads to the former
via symbolic codings.

Several higher-dimensional systems are also known to display symmetry
rigidity, as can be seen for the chair tiling \cite{Olli} or extracted
from the general theory in \cite{KS,BS} for certain dynamical systems
of algebraic origin.  In particular, there are examples from the
latter class that are not minimal.  Here, due to the action of
$\ZZ^{d}$ with $d>1$, one first needs to generalise the concept of a
reversing symmetry group. Since one now has more than one shift
generator, it seems more natural to speak of \emph{extended
  symmetries} rather than reversing symmetries. Extended symmetries
are self-homeomorphisms which give an orbit equivalence with constant
orbit cocycles. Unlike in the one-dimensional case, there are now
\emph{infinitely many} possibilities for how a homeomorphism of a
shift can shuffle orbits. This relates to $\GL(d,\ZZ)$ being an
infinite group when $d\geqslant 2$.  This setting will also be
considered, where we develop some basic theory. In view of
\cite{K-book}, it is not surprising that various new phenomena become
possible for $\ZZ^{d}$-shifts. In Theorems~\ref{thm:chair} and
\ref{thm:L}, we show that the extended symmetry groups of two classic
higher-dimensional shifts, namely the chair tiling and Ledrappier's
shift, are rather limited, reflecting the rigidity that already exists
in the case of their symmetry groups. For our proofs, we use the
techniques developed earlier in the paper, but we also need new
approaches to rule out all but finitely many possible conjugation
actions.\smallskip

The article is organised as follows. First, we discuss possible
structures of reversing and extended symmetry groups of shift spaces
through a variety of tools, which are illustrated with a substantial
selection of examples. This way, we put some emphasis on how to
actually determine the reversing or extended symmetries while, at the
same time, preparing the way for a more systematic treatment in the
future. To this end, we first recall some basic notions and tools for
the standard situation (Section~\ref{sec:prelim}), which is followed
by results on shifts in one dimension in Sections~\ref{sec:binary} and
\ref{sec:more}.  The examples are selected to cover the `usual
suspects' and to develop some intuition in parallel to the general
(and sometimes rather formal) tools. Next, we introduce and develop an
appropriate extension of the (reversing) symmetry group for
higher-dimensional shifts spaces (Section~\ref{sec:extend}). This also
establishes a connection with classic methods and results from
higher-dimensional crystallography, through the group $\GL(d,\ZZ)$ of
invertible integer matrices and their maximal finite subgroups. In
this context, we refer to \cite{TAO} for general background on notions
and results on tilings and model sets.  In Section~\ref{sec:plane}, we
apply our general results, to reach Theorems \ref{thm:chair} and
\ref{thm:L}, where we prove that two paradigmatic higher-dimensional
shifts exhibit an extended symmetry rigidity. Here, one is of
geometric and the other of algebraic origin.

\section{Setting and general tools}\label{sec:prelim}

The main concepts around reversibility are well known, and have
recently appeared in textbook form in \cite{OS}. In the context of
ergodic theory and symbolic dynamics, we also refer to \cite{LPS,GLR}
and references therein for related work and results.  In this section,
we consider symbolic dynamics in one dimension, with some emphasis on
\emph{deterministic} systems such as those defined by primitive
substitutions or by repetitive (or uniformly recurrent) sequences. 
The situation for symbolic dynamics in higher dimensions is more 
complex, both conceptually and mathematically, and thus postponed 
to Section~\ref{sec:extend}.

Let $\cA = \{ a^{}_{1}, \ldots , a^{}_{n} \}$ be a finite alphabet,
and $\XX_{\cA} := \cA^{\ZZ}$ the full shift space, or full shift for
short, which is compact in the standard product topology. Elements are
written as bi-infinite (or two-sided) sequences,
$x=(x_{i})^{}_{i\in\ZZ}$.  On $\XX_{\cA}$, we have a continuous action
of the group $\ZZ$ via the \emph{shift operator} $S$ (simply called
shift from now on) as its generator, where $(Sx)_{i} := x_{i+1}$.  We
denote the group generated by $S$ as $\langle S \ts \rangle$, and
similarly for other groups. Clearly, $(\XX_{\cA},\ZZ)$ is a
topological dynamical system, and $(S^{-1}x)_{i} = x_{i-1}$ defines
the inverse of $S$; we refer to \cite{LM} for general background. Let
$\XX$ be a closed (hence compact) and shift invariant subspace of
$\XX_{\cA}$. Below, we will always assume that the shift action on
$\XX$ is \emph{faithful}, meaning that
$\cG = \langle S|^{}_{\XX} \ts \rangle \simeq \ZZ$. This implies $\XX$
to be an infinite set. In particular, it excludes orbit closures of
periodic sequences. The corresponding dynamical system $(\XX,\ZZ)$,
also written as $(\XX,S)$, is referred to as a \emph{shift with
  faithful action}. When the context is clear, we will often simply
write $\XX$, and drop the word faithful.

Following \cite{K-book}, $\Aut (\XX)$ is the group of all
homeomorphisms of the topological space $\XX$, which is the
automorphism group in the sense of Smale (and generally much larger
than the group often defined in one-dimensional symbolic dynamics).
Clearly, one always has $S\in\Aut (\XX)$.  The \emph{symmetry group}
of $(\XX,S)$ is the (topological) centraliser of
$\langle S \ts \rangle$ in $\Aut (\XX)$, hence
\begin{equation}\label{eq:def-S}
   \cS (\XX) \, = \, \{ H \in \Aut (\XX) \mid 
     H \circ S = S \circ H \ts \} \, = \,
   \mathrm{cent}^{}_{\Aut (\XX)} \langle S \ts \rangle \ts .
\end{equation}
This is often also called the automorphism group of a shift space
\cite{Hed,CQY,DDMP,CK0}, written as $\Aut (\XX, S)$, but we shall not
use this terminology below to avoid misunderstandings with
$\Aut (\XX)$.  It is obvious that $\cS (\XX)$ will contain
$\langle S \ts \rangle = \{ S^{n} \mid n\in\ZZ \}$ as a subgroup,
where we generally assume $\langle S \ts \rangle \simeq \ts \ZZ$ as
explained earlier. However, $\cS (\XX)$ will generally be much larger
than its normal subgroup $\langle S \ts \rangle$. In fact, $\cS (\XX)$
will generally not even be amenable, which is the main reason why
$\cS (\XX)$ has not attracted much attention in the past. Yet, under
certain geometric or algebraic constraints, $\cS (\XX)$ becomes
tractable, which are the cases of most interest to us; see
Remark~\ref{rem:symm} below for more.  Moreover, Hochman's work
describes how, generically, there is only one Cantor aperiodic system;
compare \cite[Thm.~1.2]{Hoch}.

Similarly, one defines
\begin{equation}\label{eq:def-R}
   \cR (\XX) \, = \, \{ H \in \Aut (\XX) \mid
    H \circ S \circ H^{-1} = S^{\pm 1} \}  \ts .
\end{equation}
This is also a subgroup of $\Aut (\XX)$, called the \emph{reversing
  symmetry group} of $(\XX, S)$, which either satisfies
$\cR (\XX) = \cS (\XX)$ or is an index-$2$ extension of $\cS(\XX)$;
see \cite{Lamb}, or \cite{OS} and references therein.  Due to our
general assumption on the faithful action of $S$, we know that $S$
cannot act as an involution on $\XX$. Then, an element of $\cR (\XX)$
that conjugates $S$ into its inverse is called a \emph{reversor}, and
$(\XX,S)$ is called \emph{reversible}; see \cite{BR-Bulletin,OS} for
general background. Reversors are always elements of even (or
infinite) order, and for one particular class of shifts that we study,
they are always involutions; see
Corollary~\ref{cor:reversorshuffle}. Note that the reversing symmetry
group is a subgroup of the normaliser of $\cS (\XX)$ in $\Aut (\XX)$,
the latter denoted by $\mathrm{norm}^{}_{\Aut(\XX)} (\cS (\XX))$,
and agrees with it in simple situations. In general, however, it
will be a true subgroup, because we do not allow for situations where
$S$ is conjugated into some other symmetry of infinite order,
say. Nevertheless, for shifts $(\XX, S)$ with faithful action, one has
$\cR (\XX) = \mathrm{norm}^{}_{\Aut (\XX)} \langle S \ts \rangle$; see
  \cite{BR-Bulletin,OS} for more complicated situations.

\begin{rem}\label{rem:symm}
  Recent results on the symmetry groups of rigid shifts extend, in a
  straightforward manner, to similar results for the reversing
  symmetry groups. For example, Cyr and Kra \cite{CK0} show that, for
  a transitive shift with complexity of subquadratic growth, the
  (possibly infinite) factor group
  $\mathcal S(\mathbb X)/ \langle S\ts \rangle$ is \emph{periodic},
  which means that each element of it is of finite order. If $H$ is a
  reversor, $H^2$ is a symmetry, so that, for any transitive
  subquadratic shift, there exist $n \in\NN$ and $k\in\ZZ$ such that
  $(H^{2})^n=S^k$. Hence,
  $\mathcal R(\mathbb X)/ \langle S\ts \rangle$ is periodic as
  well. When the complexity of a minimal shift grows at most linearly,
  $\mathcal S(\mathbb X)/ \langle S\rangle$ is finite; see
  \cite{CK,CQY,DDMP}.  The key common point in all three proofs is
  that the set of nontrivial right-asymptotic orbit equivalence
  classes for shifts of at most linear complexity is \emph{finite};
  see below for definitions. These classes are then mapped to each
  other by a symmetry. In the same vein, a reversor maps nontrivial
  right-asymptotic orbit equivalence classes to nontrivial
  left-asymptotic classes, of which there are also only finitely many.
  \exend
\end{rem}

From now on, we will usually drop the composition symbol for
simplicity.  For the full shift $\XX_{\cA}$, the \emph{reflection} $R$
defined by
\begin{equation}\label{eq:reflect}
     (R\ts x)^{}_{n} \, = \, x^{}_{-n}
\end{equation}
is an involution in $\Aut (\XX_{\cA})$, and it is easy to check that
$RS\nts R = S^{-1}$. Consequently,
\[
   \cR (\XX_{\cA}) \, = \, \cS (\XX_{\cA}) \rtimes C_{2}
\]
is a semi-direct product by standard arguments \cite{BR-Bulletin,OS},
where $C_{2} = \langle R \ts \rangle$ is the cyclic group of order $2$
that is generated by $R$. This will also be the situation for all
shift spaces $\XX\subseteq \XX_{\cA}$ that are invariant under
$R$. Following \cite{TAO}, we call a shift space $\XX$ reflection
symmetric or \emph{reflection invariant} if $R ( \XX ) = \XX$. Note
that this property only refers to $\XX$ as a (topological) space, and
makes no reference to $S$. The following standard result is a
consequence of \mbox{\cite[Lemma~1]{BR-Bulletin}} and the structure of
semi-direct products.

\begin{fact}\label{fact:semi}
  If\/ $(\XX, S)$ is a reflection invariant shift such that\/ $S$ does
  not act as an involution, the reversing symmetry group satisfies
\[
   \cR (\XX) \, = \, \cS (\XX) \rtimes C_{2}
\]
  with\/ $C_{2} = \langle R \ts \rangle$ and\/ $R$ as defined in
  Eq.~\eqref{eq:reflect}, which is thus an involutory reversor. 

  More generally, irrespective of reflection invariance of\/ $\XX$,
  the same structure of\/ $\cR (\XX)$ emerges if\/ $\Aut (\XX)$
  contains an involution, $R^{\ts \prime}$ say, which
  conjugates\/ $S$ into its inverse.

  If, in addition, one has\/
  $\cS (\XX) = \langle S \rangle \simeq \ZZ$, any reversor must be of
  the form\/ $R^{\ts \prime} \ts S^{m}$ for some\/ $m\in\ZZ$, and is
  thus an involution as well.  \qed
\end{fact}

Let us return to the full shift and ask for other reversors. Clearly,
we could also consider another reflection, $R^{\ts \prime}$, defined by
\begin{equation}\label{eq:def-alt-R}
   (R^{\ts \prime} x)^{}_{n} \, = \, x^{}_{-n-1} \ts .
\end{equation}
This is another involution, related to the previous one by
$R^{\ts \prime} = SR$, which will show up many times later on. More
generally, any reversor for $\XX_{\cA}$ can be written as $G R$ with
$R$ from Eq.~\eqref{eq:reflect} and $G\in\cS (\XX_{\cA})$; this is
described in the proof of Proposition~\ref{prop:CHL}. Now, the full
shift is certainly invariant under any map $F_{\alpha}$ that is
induced by applying the same permutation
$\alpha \in \varSigma^{}_{\! \cA}$ to each entry of a sequence $x$,
\begin{equation}\label{eq:def-LEM}
    (F_{\alpha} (x))_{n} \, := \, \alpha (x_{n}) \ts ,
\end{equation}
where $\varSigma^{}_{\! \cA}$ denotes the permutation group of
$\cA$. We call such a map a \emph{letter exchange map}, or LEM for
short. Now, for $\XX_{\cA}$, we clearly also have the reversors
$R_{\alpha}$, defined by
\begin{equation}\label{eq:perm-rev}
     (R_{\alpha} ( x ))^{}_{n} \, := \, \alpha (x^{}_{-n}) \ts ,
\end{equation}
which includes our original reversor via $R=R_{e}$ with $e$ the
identity in $\varSigma^{}_{\! \cA}$. Similarly, one may consider
$R^{\ts\prime}_{\alpha} = F_{\alpha} R^{\ts \prime} = S R_{\alpha}$.  Clearly,
$F_{\alpha}$ commutes with both $R$ and $R^{\ts \prime}$. Note also
that $R_{\alpha}$ need no longer be an involution, though it is always
an element of even order, which is
\[
   \ord (R_{\alpha} ) \, = \, \lcm (2, \ord (\alpha)) \ts . 
\]   
These letter permuting reversors form an interesting class of
candidates for shifts that fail to be reflection invariant, as we
shall see later. Clearly, if $R_{\alpha}$ is a reversor, then so is
$R^{\ts\prime}_{\alpha}=SR_{\alpha}$ as well as $S^{m}R_{\alpha}$ for any
$m\in\ZZ$, which all have the same order as $R_{\alpha}$.  \smallskip

Fact~\ref{fact:semi} reveals nothing about the structure of
$\cS (\XX)$, but some classic results \cite{Coven} together with
recent progress \cite{Olli,CQY} allow the determination of this group
for many interesting cases.  Also, Fact~\ref{fact:semi} indicates that
particularly noteworthy instances of reversibility will only emerge
for shifts that \emph{fail} to be reflection invariant. So, let us
expand a little on reflection invariance.

Recall that a shift $\XX$ is called \emph{minimal} when every
(two-sided) shift orbit in $\XX$ is dense.  A minimal shift $\XX$ is
called \emph{palindromic} when it contains an element $x$ that
satisfies $R(x)=x$ (odd core) or $R^{\ts \prime}(x)=x$ (even core).
Palindromicity of $\XX$ then implies its reflection invariance. The
converse is also true for substitution shift spaces over binary
alphabets \cite{Tan}, but false for general shifts \cite{BBCF}, and
seems still open for substitution shifts \cite{Q} on a larger alphabet.

Consider a shift $\XX$ that is defined by a primitive substitution
rule $\theta$, as the orbit closure of a fixed point of $\theta^{n}$
for some $n\in\NN$, say, where here and below we use the
approach from \cite{TAO}. Note that the notion of a two-sided fixed
point $x$ includes the condition that the core or \emph{seed} of $x$,
which is the word $x^{}_{-1} | x^{}_{0}$, is legal for $\theta$.
Since any such fixed point is repetitive, $\XX$ is minimal, which is a
consequence of Gottschalk's theorem; see \cite{Pet,TAO} for
details. Now, we call $\theta$ \emph{palindromic} if the shift $\XX$
defined by $\theta$ is; compare \cite[Sec.~4.3]{TAO}. For the
palindromicity of $\theta$, a useful sufficient criterion was
formulated in \cite[Lemma~3.1]{HKS}; see also \cite[Lemma~4.5]{TAO}.

\begin{fact}\label{fact:palin}
  Let\/ $\theta$ be a primitive substitution on\/
  $\cA=\{a^{}_{1},\ldots,a^{}_{n}\}$, with\/
  $\theta(a^{}_{i}) = p\ts q^{}_{i}$ for\/ $1\leqslant i\leqslant n$,
  where\/ $p$ and all\/ $q^{}_{i}$ are palindromes, possibly
  empty. Then, $\theta$ is palindromic.  Likewise, $\theta$ is
  palindromic if\/ $\theta(a^{}_{i}) = q^{}_{i} \ts\ts p$ with\/ $p$
  and all\/ $q^{}_{i}$ being palindromes, possibly empty.  \qed
\end{fact}

Let us now discuss some general tools that we will need in our further
considerations. We first recall the classic Curtis--Hedlund--Lyndon
(CHL) theorem \cite[Thm.~6.2.9]{LM} in a version for symmetries of
shift spaces, as well as the corresponding one for reversors, whose
proof is similar to that of the classical version; see \cite{LM} for
background.

\begin{prop}\label{prop:CHL}
  Let\/ $\XX$ be a shift over the finite alphabet\/ $\cA$, with
  faithful shift action. For any\/ $G\in \cS (\XX)$, there exist
  non-negative integers\/ $\ell , r$ together with a map\/
  $h \! : \, \cA^{\ell+r+1}\xrightarrow{\quad} \cA$ such that\/
  $(G(x))^{}_{n} = h (x^{}_{n-\ell}, \ldots ,x^{}_{n}, \ldots
  ,x^{}_{n+r})$
  holds for all\/ $x\in \XX$ and\/ $n\in \ZZ$. Likewise, for any
  reversor\/ $H\in \cR(\XX) \setminus \cS(\XX)$, there are
  non-negative integers\/ $\ell , r$ and a map\/ $h$ such that\/
  $(H(x))^{}_{n}= h (x^{}_{-n-r}, \ldots ,x^{}_{-n}, \ldots ,
  x^{}_{-n+\ell})$.
\end{prop}

\begin{proof}
If $G\in \cS (\XX)$, we know that $G\in \Aut (\XX)$ and the 
(horizontally written) diagram
\[
\begin{CD}
   \XX @> G >> \XX \\
   @V S VV   @VV S V \\
   \XX @> G >> \XX
\end{CD}
\]
commutes, so the classic CHL theorem applies and asserts that $G$ is a
sliding block map with \emph{local rule}, or \emph{local derivation
  rule}, $h$.

Now, let $H$ (and hence also $H^{-1}$) be a reversor. Since 
the involution $R$ from Eq.~\eqref{eq:reflect} need not be
in $\Aut (\XX)$, we define $\YY = R (\XX)$ and consider the 
commutative diagram
\[
\begin{CD}
   \XX @> R >> \YY @> HR >> \XX \\
   @V S^{-1} = H^{-1} S H VV   @V T VV @ VV S V \\
   \XX @> R >> \YY @> HR >> \XX
\end{CD}
\]
where $T$ is the shift on $\YY$, defined the same way as $S$. 
Applying the CHL theorem to $H' =
H R$ gives us two integers $\ell, r \geqslant 0$ and some function
$h'$ such that $(H' y)^{}_{n} = h' (y^{}_{n-\ell}, \ldots , y^{}_{n+r}
)$ for all $y\in\YY$. Now, set $h (z^{}_{1}, \ldots , z^{}_{\ell + r +
  1} ) = h' (z^{}_{\ell + r + 1}, \ldots, z^{}_{1} )$ and observe that
each $y\in\YY$ can be written as $y = Rx$ for some $x\in\XX$,
wherefore we get
\[
\begin{split}
   (H x)^{}_{n} \, & = \,  \bigl(H \! R (Rx)\bigr)_{n} \, = \,
   \bigl( H' y \bigr)_{n} \, = \, 
    h' ( y^{}_{n-\ell}, \ldots , y^{}_{n+r} ) \\
    & = \, h' ( x^{}_{\ell - n}, \ldots , x^{}_{-r -n} ) \, = \,
   h (x^{}_{-n-r}, \ldots , x^{}_{-n+\ell} )
\end{split}   
\] 
as claimed.
\end{proof}

In line with \cite[Sec.~5.2]{TAO}, if $G\in\cS(\XX)$ has a local
derivation rule $h$ that is specified on
$(\ell \! + \! r \! + \! 1)$-tuples of letters in $\cA$, we call
$\max \{\ell, r \}$ the \emph{radius} of $G$.  Further, if $H$ is a
reversor, it can be written as $H=H' R$ with
$H'\! : \, R(\XX) \xrightarrow{\quad} \XX$ some sliding block map.  We
will frequently use the local rule of $H'$ to investigate $H$.  In the
proof of Proposition \ref{prop:CHL}, any involution could have been
used, and indeed, sometimes it will be more convenient for us to write
$H$ as $H= (H R^{\ts \prime}) R^{\ts \prime}$ and to work with the the
local rule of $H' :=HR^{\ts \prime}$ instead. These two approaches
reflect the two possible types of (infinite) palindromes, namely those
with odd or with even core.

Let us next recall the notion of the maximal equicontinuous factor,
followed by the introduction of some important tools connected with
it.  To this end, let $(\XX,S)$ be a topological dynamical system. The
dynamical system $(\AAA,T)$ is then called the \emph{maximal
  equicontinuous factor} (MEF) of $(\XX,S)$ if $(\AAA,T)$ is an
equicontinuous factor of $(\XX,S)$, which means that it is a factor
with the action of $T$ being equicontinuous on $\AAA$, with the
property that any other equicontinuous factor $(\YY,R)$ of $(\XX,S)$
is also a factor of $(\AAA,T)$.  The MEF of a minimal transformation
is a rotation on a compact monothetic topological group
\cite[Thm.~2.11]{Pet}, which is a group $\AAA$ for which there exists
an element $a$ such that the subgroup generated by $a$ is dense. Such
a group is always Abelian and we will thus write the group operation
additively.

The setting is that of a commutative diagram,
\begin{equation}\label{eq:CD1}
\begin{CD}
   \XX @> S >> \XX \\
   @V \phi VV   @VV \phi V \\
   \AAA @> T >> \AAA
\end{CD}
\end{equation}
with $\phi$ a continuous surjection that satisfies
$\phi \circ S = T \circ \phi$ on $\XX$. Also, $T$ is realised as a
translation on $\AAA$, so there is a unique $a\in\AAA$ such that
$T (y) = y+a$ for all $y\in\AAA$.  The MEF can be trivial, but many
shifts possess a nontrivial MEF, and if so, we can glean information
from it about the symmetry and reversing symmetry groups. This is
described by the following theorem, the proof of which is an extension
of \cite[Thm.~3.3]{CQY} to include reversing symmetries.

\begin{thm}\label{thm:equicontinuous}
  Let\/ $(\XX,S \ts )$ be a shift with faithful shift action and at
  least one dense orbit. Suppose further that the group rotation\/
  $(\AAA, T)$ with dense range is its MEF, with\/ $\phi \! : \, \XX
  \xrightarrow{\quad} \AAA$ the corresponding factor map.

  Then, there is a group homomorphism\/ $\kappa \! : \, \cS (\XX)
  \xrightarrow{\quad} \AAA$ such that 
\[
   \phi (G(x)) \, = \, \kappa(G) + \phi (x)
\] 
holds for all\/ $x\in \XX$ and\/ $G \in \cS (\XX)$, so any symmetry\/
$G$ induces a unique mapping\/ $G^{}_{\! \AAA}$ on the Abelian group\/
$\AAA$ that acts as an addition by\/ $\kappa (G)$.

Moreover, if\/ $(\XX, S \ts )$ is a reversible shift,
there is an extension of\/ $\kappa$ to a\/ $1$-cocycle of the
action of\/ $\cR (\XX)$ on\/ $\AAA$ such that
\begin{equation}\label{eq:cocycle}
   \kappa(G H) \, = \, \kappa(G) + \varepsilon(G) \, \kappa(H)
\end{equation}
for all\/ $G,H \in \cR (\XX)$, where\/
$\varepsilon\! : \, \cR(\XX) \xrightarrow{\quad} \{ \pm 1 \}$ is a
group homomorphism with values\/ $\varepsilon(G)=1$ for symmetries
and\/ $\varepsilon(G)=-1$ for reversing symmetries.\footnote{As such,
$\varepsilon$ can be viewed as a homomorphism into\/
$\Aut (\AAA)$ with kernel\/ $\cS (\XX)$.}

As a consequence, one has\/ $\kappa (H^{2}) = 0$ for all\/
$H\in\cR(\XX) \setminus \cS(\XX)$. Any such reversor\/ $H$ induces a
mapping\/ $H^{}_{\! \AAA}$ on\/ $\AAA$ that acts as\/
$z \mapsto \kappa (H) - z$, hence
\[
  \phi (H(x)) \, = \, \kappa(H) - \phi (x) \ts ,
\]
  for all\/ $H\in\cR(\XX) \setminus \cS(\XX)$ and all\/ $x\in\XX$.
\end{thm}

\begin{proof}
  Let $x \in \XX$ be a point with dense shift orbit, which exists by
  assumption. Due to the property of the MEF, we know that $T$ in
  Eq.~\eqref{eq:CD1} acts as addition with dense range, so there is an
  $a\in\AAA$ such that, for all $z\in \AAA$, $T(z) = z + a$ with $\{ z
  + n a \mid n \in \ZZ \}$ dense in $\AAA$.

  Fix an arbitrary symmetry $G \in \cS(\XX)$. We need to show that
  there is a unique induced mapping $G^{}_{\! \AAA}$ on the MEF, via
  the corresponding commutative diagram, and determine what this mapping
  is.  To this end, consider
\[
    f(x) \, := \, \phi ( Gx ) - \phi (x) \ts ,
\]
which is an element of $\AAA$. Since
$\phi (S^{n} x) = \phi (x) + n \ts a$ and
$\phi( G S^{n} x) = \phi(S^{n} G x) = \phi( G x) + n \ts a$, one sees
that $f(S^{n} x ) = f(x)$ for $n\in\ZZ$ is the natural extension of
$f$ to a function on the orbit of $x$. For any $y\in\XX$, there is a
sequence $\bigl( x^{(i)}\bigr)_{i\in\NN}$ with $x^{(i)} = S^{n_{i}} x$
and $x^{(i)} \xrightarrow{\, i \to \infty \, } y$, so that
$f(y) := f(x)$ is the unique way to extend $f$ to a continuous function
on $\XX$ with values in $\AAA$, which is constant in this case.
Consequently, the mapping
$G^{}_{\! \AAA} \! : \, \AAA \xrightarrow{\quad} \AAA$ induced on the
factor by $G$ acts as a translation by
\[
    \kappa (G) \, := \, \phi ( Gx ) - \phi (x) \ts .
\]
This defines a mapping $\kappa \! : \, \cS(\XX) \xrightarrow{\quad}
\AAA$, with $\kappa (S^{n}) = n \ts a$.  The homomorphism property
$\kappa (G H ) = \kappa(G) + \kappa(H)$ then follows from the
concatenation of two commutative diagrams.

Now, if $R\in\cR(\XX) \setminus \cS(\XX)$ is a reversor, one can modify
the derivation, this time leading to
\[
     \kappa (R) \, := \, \phi (Rx) + \phi (x)
\]
because
$\phi (R S x) = \phi (S^{-1} R x) = \phi (R x) - \kappa (S) = \phi (R
x) - a$,
hence $\phi (R S^{n} x) + \phi (S^{n} x) = \phi (R x) + \phi (x)$ for
all $n\in \ZZ$.  As a result, one has
$\phi (R y) = \kappa (R) - \phi(y)$ for all $y \in \XX$ as claimed,
which means that the action of
$R^{}_{\AAA} \! : \, \AAA \xrightarrow{\quad} \AAA$ induced by $R$ is
consistently given by $z \mapsto - z + \kappa (R)$.  The cocycle
property is then a consequence of another concatenation of commuting
diagrams for the different possibilities, and $\kappa (R^2) = 0$
follows immediately.
\end{proof}

A closer inspection of the proof shows that it does not need the
maximality property of the MEF, so that a corresponding result also
holds for general equicontinuous factors. However, since the MEF is
the most useful among them for our porposes, we skip further details;
compare \cite[Sec.~6]{Aus} for related material.  Let us formulate one
important consequence.

\begin{coro}\label{coro:equicontinuous}
  Let the setting for a shift\/ $(\XX, S)$ be that of
  Theorem~$\ref{thm:equicontinuous}$, and define the covering number\/
  $c := \min \{ \card(\phi^{-1}(z)) \mid z \in \AAA \}$.  If\/ $c$ is
  finite, one has that
\begin{enumerate}\itemsep=2pt
\item for each\/ $d\geqslant c$,
  $\{ z \in \AAA \mid \card(\phi^{-1}(z)) =d \} = \{z \in \AAA \mid
  \card(\phi^{-1}(z)) =d \} + \kappa(G)$
  holds for any\/ $G\in \cS(\XX)$,
\item for each\/ $d\geqslant c$,
  $\{ z \in \AAA \mid \card(\phi^{-1}(z)) =d \} = -\{z \in \AAA \mid
  \card(\phi^{-1}(z)) =d \} + \kappa(G)$
  holds for each\/ $G\in \cR (\XX) \backslash \mathcal S(\XX)$, and
\item $\kappa \! : \,  \cS(\XX) \xrightarrow{\quad}  
  \AAA$ and $\kappa \! : \, \cR (\XX) \setminus \cS(\XX) 
  \xrightarrow{\quad}  \AAA$ are each at  most\/ $c$-to-one. 
\end{enumerate} 
\end{coro}

\begin{proof}
These claims are an obvious generalisation of \cite[Thm.~3.3]{CQY}.
\end{proof}
 
We refer to \cite{CQY} for details on how to compute $c$ and $\{ z \in
\AAA \mid \card(\phi^{-1}(z)) > c \}$ for the important class of shifts
that are generated by constant-length substitutions.
\smallskip

Let $(\XX,S)$ be a minimal two-sided shift, and let $\dX$ be a metric
on $\XX$ that induces the local topology on it. Now, we define
equivalence relations on orbits in $\XX$ as follows, where we write
orbits as $\cO (x) = \{S^n (x) \mid n\in \ZZ\}$ for $x\in\XX$.  Two
orbits $\cO_{x}$ and $\cO_{y}$ are \emph{right-asymptotic}, denoted
$\cO_{x} \stackrel{r}{\sim} \cO_{y}$, if there exists an $m\in\ZZ$
such that $\lim_{n\to\infty} \dX (S^{m+n}x,S^ny) = 0$. The right
asymptotic equivalence class of $\cO_x$ is denoted $[x]^{}_{r}$. A
right-asymptotic equivalence class will be called \emph{non-trivial}
if it does not consist of a single orbit.  Analogously, we define the
notion of \emph{left-asymptotic} orbits, with accompanying equivalence
relation $\stackrel{\ell}{\sim}$ and class $[x]^{}_{\ell}$. The
following connection with reversibility is clear.

\begin{fact}\label{fact:classes}
  Every reversor of a minimal two-sided shift\/ $(\XX, S)$ must map
  any non-trivial right-asymptotic class to a left-asymptotic one
  of the same cardinality, and vice versa.  \qed
\end{fact}

We know from Theorem~\ref{thm:equicontinuous} that $\kappa (H^{2}) = 0$
for any reversor $H$. Moreover, one has the following result.

\begin{coro}\label{cor:reversorshuffle}
  Under the assumptions of Theorem~$\ref{thm:equicontinuous}$, one has
  the following properties.
\begin{enumerate}\itemsep=2pt
\item   Any symmetry that is the square of a reversor lies in the kernel
  of\/ $\kappa$.
\item If\/ $\kappa$ is injective on\/ $\cS (\XX)$, every reversor is an
  involution.
\item If the MEF of\/ $(\XX,S)$ is non-trivial and if\/ $\cS (\XX) =
  \langle S \ts \rangle \simeq \ZZ$, every reversor is an involution.
\end{enumerate}
\end{coro}

\begin{proof}
  If $H$ is a reversor, $H^{2}$ is a symmetry, with $\kappa (H^{2})=0$
  by Theorem~\ref{thm:equicontinuous}. Any $G\in\cS (\XX)$ with $G =
  H^{2}$ for some reversor $H$ must then satisfy $\kappa (G)=0$.

  Now, if $\kappa$ is injective on $\cS (\XX)$, its kernel is trivial,
  so $H^{2} = \id$ for any reversor $H$. Similarly, if $\cS (\XX) =
  \langle S \ts \rangle$, one has $H^{2} = S^{m}$ for some $m\in\ZZ$,
  hence $0 = \kappa (H^{2}) = m\ts \kappa(S)$. When the MEF is
  non-trivial, we know that $\kappa (S) \ne 0$, hence $m=0$ and $H^{2}
  = \id$.
\end{proof}

\section{Binary shifts}\label{sec:binary}

Let us begin with the class of Sturmian shifts over the binary
alphabet $\cA = \{ a,b \}$. Following \cite{CH} and
\cite[Def.~4.16]{TAO}, we call a two-sided sequence $x \in \cA^{\ZZ}$
\emph{Sturmian} if it is nonperiodic, repetitive and of minimal
complexity, where the latter means that $x$ has word complexity 
$p(n)= n+1$ for all $n\in\NN_{0}$. The corresponding shift 
$\XX_{x}$ is the orbit closure of $x$ under the shift action, so
\[
      \XX_{x} \, = \, \overline{ \{ S_{\vphantom{I}}^{n} 
           (x) \mid n\in\ZZ \} }
\] 
with the closure being taken in the product topology, which is also
known as the local topology; compare \cite[p.~70]{TAO}. It is clear
from the definition that $\XX_{x}$ is minimal and aperiodic, hence the
shift action on it is faithful.

\begin{thm}\label{thm:Sturmian}
  Let\/ $x$ be an arbitrary two-sided Sturmian sequence, and\/
  $\XX_{x}$ the corresponding shift. Then,
  $\cS (\XX_{x}) = \langle S \ts \rangle \simeq \ts \ZZ$ is its
  symmetry group. Moreover, the shift is reversible, with\/
  $\cR (\XX_{x}) = \cS (\XX_{x}) \rtimes \langle R \ts \rangle \simeq
  \ts \ZZ \rtimes C_{2}$
  and\/ $R$ the reflection from Eq.~\eqref{eq:reflect}. In particular,
  all reversors are involutions.
\end{thm}

\begin{proof}
  For any Sturmian sequence $x$, we know
  $\cS (\XX_{x}) \simeq \ts \ZZ$ from \cite[Sec.~4]{Olli};
  alternatively, see \cite[Cor.~5.7]{CK} or \cite[Thm.~3.1]{DDMP}.
  Moreover, $\XX_{x}$ is \emph{always} palindromic as a consequence of
  \cite[Prop.~6]{DP}, so we have reflection invariance in the form
  $R (\XX_{x}) = \XX_{x}$ with $R$ from Eq.~\eqref{eq:reflect}, and
  thus, by Fact~\ref{fact:semi},
  $\cR (\XX_{x}) \simeq \ts \ZZ \rtimes C_{2} $ as stated. The last
  claim also follows from Fact~\ref{fact:semi}.
\end{proof}

Next, we look into binary shifts that are generated by an aperiodic
primitive substitution.  Recall that a shift is called
\emph{aperiodic} if it does not contain any element with a non-trivial
period under the shift action.  A primitive substitution rule is
aperiodic when the shift defined by it is; compare
\cite[Def.~4.13]{TAO}.

\begin{example}\label{ex:Fibo}
  Consider the primitive substitution 
\[
    \theta^{}_{\mathrm{F}} : \quad 
    a\mapsto a b \ts , \, b\mapsto a \ts , 
\]    
which is known as the \emph{Fibonacci substitution}; see
\cite[Ex.~4.6]{TAO}. More generally, let $m\in \NN$ and consider the
primitive substitution
\[
   \theta^{}_{m} : \quad a\mapsto a^{m}  b \ts , \,  b\mapsto a \ts ,
\]   
which is a \emph{noble means substitution}, with $\theta^{}_{1} =
\theta^{}_{\mathrm{F}}$; compare \cite[Rem.~4.7]{TAO}. Each
$\theta^{}_{m}$ is aperiodic and defines a minimal shift $\XX_{m}$,
hence with faithful shift action. In fact, any such shift is
Sturmian, so Theorem~\ref{thm:Sturmian} gives us
\[
    \cS (\XX_{m}) \, = \, \langle S \ts \rangle
    \, \simeq \, \ZZ  \quad \text{and} \quad
   \cR (\XX_{\mathrm{F}}) \, = \, \langle S \ts \rangle
   \rtimes \langle R \ts \rangle \, \simeq \, \ZZ \rtimes C_{2}\ts ,
\]
with $R$ the reversor from Eq.~\eqref{eq:reflect}. This is the
simplest possible situation with reversibility; compare
Fact~\ref{fact:semi} and \cite[Thm.~1]{BR-Bulletin}.  In particular,
\emph{all} reversors are involutions.

Let us mention an independent way to establish reversibility via
palindromicity directly on the basis of the substitution, without
reference to \cite{DP}.
This is clear for $\theta^{}_{\mathrm{F}}$, which is
palindromic by Fact~\ref{fact:palin}. More generally, one can use
\cite[Prop.~4.6]{TAO} and observe that, for each $m$, the shift
$\XX_{m}$ can also be defined by a different substitution
$\theta^{\ts \prime}_{m}$ that is conjugate to $\theta^{}_{m}$. If
$m = 2\ell$, we use $a\mapsto a^{\ell} b a^{\ell}$ and $b\mapsto a$,
which satisfies the conditions of Fact~\ref{fact:palin} with $p$ the
empty word. If $m = 2\ell+1$, we use $a\mapsto a^{\ell+1} b a^{\ell}$
and $b\mapsto a$, so that Fact~\ref{fact:palin} applies with $p=a$ but
$q^{}_{b}$ empty.  Thus, all $\theta^{\ts \prime}_{m}$ are palindromic
and all $\XX_{m}$ reflection invariant.  

Finally, let us also mention that the noble means hulls can be
described as \emph{regular model sets}; see \cite[Rem.~4.7, Sec.~7.1
and Ex.~7.3]{TAO} for details. The corresponding cut and project
schemes are Euclidean in nature, both for the direct and the internal
space.  As a consequence, our results can alternatively be derived by
means of the corresponding \emph{torus parametrisation} \cite{BHP},
which adds a more geometric interpretation of the MEF in these cases.
\exend
\end{example} 

Let us comment on an aspect of constant length substitutions that can
simplify the determination of reversors. We shall use this property
later on.

\begin{rem}    
  Let $\theta$ be a primitive constant-length substitution, of
  \emph{height $h$} and \emph{pure base} $\theta'$ (see
  \cite{Dekking,Q} for background), with corresponding shift spaces
  $\XX=\XX_{\theta}$ and $\YY=\XX_{\theta'}$. If the height $h$ of
  $\theta$ is nontrivial, which means $h>1$, one can recover
  $\cR(\XX_\theta)$ from $\cR(\XX_{\theta'} )$. This is so because
  $\XX$ can be written as a tower of height $h$ over $\YY$, so $\XX$
  is conjugate to $\{ (y,i) \mid y\in\YY , \, 0 \leqslant i < h \}$.
  Now, the arguments of \cite[Prop.~3.5]{CQY} give the symmetries of
  $(\XX,S)$ and extend to show that any reversor of $(\XX, S)$ must be
  of the form
\[
    H^{}_{i}\ts (y,j) \, := \,
    \begin{cases}
       \bigl(H'(y), i-j \bigr)  , & 
              \text{if } i-j \geqslant 0 \ts , \\
       \bigl((T^{-1}H' )(x), i-j \bmod h\bigr)  , &
              \text{if } i-j < 0 \ts ,
\end{cases}
\]
for some $H' \in \cR(\YY)\setminus \cS(\YY)$ and some $0\leqslant i <
h$. In this sense, it suffices to analyse the height-$1$
substitutions.  \exend
\end{rem}

A height-$1$ substitution $\theta$ of constant length $r$ is said to
have \emph{column number} $c^{}_{\theta}$ if, for some $k\in \NN$ and
some position numbers $(i^{}_1, \ldots ,i^{}_k)$, one has
$\card \bigl(\theta_{i^{}_1}\!\circ \ldots
\circ\theta_{i^{}_k}(\mathcal A)\bigr)=c^{}_{\theta}$
with $c^{}_{\theta}$ being the least such number.\footnote{For shifts
  defined by substitutions of constant length, it follows from
  \cite{CQY} that $c^{}_{\theta} = c,$ where $c$ is the covering
  number from Corollary \ref{coro:equicontinuous}.} Here,
$\theta_{i} (a)$ with $1 \leqslant i \leqslant r$ and $a\in\cA$
denotes the value of $\theta(a)$ at position $i$.  Note that, if $\cA$
is a binary alphabet, one can use $k = 1$ in the above definition.  In
particular, if $\theta$ has column number $1$, we will say that
$\theta$ has a \emph{coincidence}. If $\theta_i$ acts as a permutation
on $\cA$ for each $1\leqslant i \leqslant r$, we call $\theta$
\emph{bijective}.  In the latter case, when on a binary alphabet
$\{ a,b\}$, $\theta$ commutes with the letter exchange
$\alpha \!  : \, a \leftrightarrow b$, wherefore the shift
$\XX_{\theta}$ is invariant under the LEM $F_{\alpha}$. Since
$F_{\alpha}$ commutes with the shift action, we have
$F_{\alpha} \in \cS(\XX_{\theta})$. The following result by Coven
\cite{Coven} is fundamental.  For convenience, we reformulate it in
the modern terminology of symbolic dynamics and substitution shifts
\cite{Q,TAO}.

\begin{lem}\label{lem:Coven}
  Let\/ $\theta$ be a primitive and aperiodic substitution rule of
  constant length over a binary alphabet, $\cA = \{ a,b \}$ say, and
  let\/ $\XX_{\theta}$ be the shift defined by it.  If\/
  $\theta$ has a coincidence, the symmetry group is\/
  $\cS (\XX_{\theta} ) = \langle S \ts \rangle \simeq \ts \ZZ$. If\/
  $\theta$ is a bijective substitution, the only additional symmetry
  is the LEM\/ $F_{\alpha}$, so\/
  $\cS (\XX_{\theta} ) = \langle S \ts \rangle \times \langle
  F_{\alpha} \ts \rangle \simeq \ts \ZZ \times C_{2}$.  \qed
\end{lem}

Let us now turn our attention to substitution shifts outside the
Sturmian class.  So far, all reversible examples were actually
reflection invariant shifts. This is deceptive in the sense that
one can have reversibility without reflection invariance.

\begin{example}\label{ex:flip-rev}
  Consider the primitive substitution rule
\[
     \theta : \quad a \mapsto aaba \ts, \; b \mapsto babb \ts ,
\]  
which is aperiodic, of constant length and height $1$; see
\cite{Dekking,Q} for background.  Since $\theta$ has a coincidence (in
the second and third position), it defines a strictly ergodic shift
$\XX_{\theta}$ with pure point dynamical spectrum. Consequently,
$\XX_{\theta}$ is one-to-one almost everywhere over its MEF, and we
also have a description as a regular model set; compare
Example~\ref{ex:pd} below. The symmetry group is minimal,
$\cS (\XX_{\theta}) = \langle S \ts \rangle \simeq \ts \ZZ$, the
latter as a consequence of Lemma~\ref{lem:Coven}.
  
Now, it is easy to see that $bbaa$ is a legal word, while $aabb$ is
not, so $\XX_{\theta}$ fails to be reflection invariant, and neither
$R$ from Eq.~\eqref{eq:reflect} nor $R^{\ts \prime}$ from
Eq.~\eqref{eq:def-alt-R} will be a reversor for $(\XX_{\theta},S)$. As
mentioned above, $F_{\alpha}$ with $\alpha$ the letter exchange
$a \leftrightarrow b$ (which is an element of the permutation group
$\varSigma^{}_{\! \cA}$ of the alphabet) is an LEM on the full shift.
However, $F_{\alpha}$ does not map $\XX_{\theta}$ into itself, as can
be seen from the same pair of words just used to exclude
$R^{\ts \prime}$.

Recall that $R^{\ts \prime}$ and $F_{\alpha}$ commute within
$\Aut (\XX_{\cA})$, and that
$R^{\ts\prime}_{\alpha} = F^{}_{\alpha} \ts R^{\ts \prime}$ is a reversor for the
full shift.  Now, $\theta$ has $4$ bi--infinite fixed points, with
legal seeds $a|a$, $a|b$, $b|a$ and $b|b$. One can check that
$R^{\ts\prime}_{\alpha}$ fixes those with seeds $a|b$ and $b|a$ individually,
while permuting the other two. Via a standard orbit closure argument,
this implies that $R^{\ts\prime}_{\alpha}$ maps $\XX^{}_{\theta}$ into itself,
and is an involutory reversor for $(\XX_{\theta},S)$.  We thus get
\[
     \cR (\XX^{}_{\theta}) \, = \, \langle S \ts \rangle \rtimes 
     \langle R^{\ts\prime}_{\alpha} \ts \rangle
     \, \simeq \, \ZZ \rtimes C_{2} \ts ,
\]  
which is an example of reversibility with a non-standard
involutory reversor, but still with the `standard' structure
of $\cR(\XX_{\theta})$ that we know form Fact~\ref{fact:semi}.
In particular, all reversors are once again involutions in this 
case.   \exend
\end{example}

Quite frequently, we need to check whether an LEM generates a symmetry
or contributes to a reversing symmetry in the sense of
Eq.~\eqref{eq:perm-rev}. To this end, the following criterion is
handy.  Its proof is a generalisation of \cite[Props.~3.19 and
3.21]{CQY}, and the remark following Corollary~3.20 in that same
article. We use the map $\kappa$ defined in Theorem
\ref{thm:equicontinuous}.

We say that $\theta$ is \emph{strongly injective} if $\theta$ is
injective on letters and does not have any right-infinite fixed points
which differ only on their initial entry, or any left-infinite fixed
points which differ only on their initial entry. A form of the
following lemma for symmetries has been formulated and proved in joint
preliminary work by the third author and A.~Quas; the authors thank
him for the permission to reproduce its proof.

\begin{lem}\label{lem:necessary}
  Let \/ $\theta$ be a primitive substitution of constant length\/ $r$
  with height\/ $1$ and column number\/ $c^{}_{\theta}$.  Suppose
  that\/ $\theta$ is strongly injective. Then, any map\/
  $G\in \cS (\XX_{\theta})$ with\/ $\kappa (G) = 0$ must have radius\/
  $0$, so that\/ $G=F_{\alpha}$ for some\/
  $\alpha\in \varSigma^{}_{\! \cA}$, the permutation group of\/ $\cA$.
  Likewise, if\/ $H\in\cR (\XX_{\theta})$ is a reversor with\/
  $\kappa (H)=0$, it must be of the form\/
  $G=F_{\alpha} \ts R^{\ts \prime}$.

  Moreover, with\/ $\cA^{2}_{\theta}$ denoting the set of\/
  $\theta$-legal words of length\/ $2$, 
   a permutation\/ $\alpha \in \varSigma^{}_{\! \cA}$  generates 
   an LEM\/ $F_{\alpha}  \in      \cS(\XX_{\theta})$ if and only if
\begin{enumerate}\itemsep2pt
   \item the permutation\/ $\alpha$ maps\/ $\cA_{\theta}^{2}$ to\/ 
        $\cA_{\theta}^{2}$, and
   \item $\bigl(\alpha \circ \theta^{c^{}_{\theta}!}\bigr)(ab) =
     \bigl(\theta^{c^{}_{\theta}!} \circ \alpha\bigr) (ab)$ for each $ab
     \in \cA_{\theta}^{2}$.
\end{enumerate}
Similarly, a permutation\/ $\alpha$ generates a reversor\/
$G= F_\alpha R^{\ts \prime} \in \cR(\XX_{\theta})\setminus
\cS(\XX_{\theta})$ if and only if
\begin{itemize}\itemsep2pt
\item[(3)] $\ts ab \in \cA_{\theta}^{2}$ implies\/
  $\alpha(ba) \in \cA_{\theta}^{2}$, and
\item[(4)]
  $\bigl(\alpha \circ \theta^{c^{}_{\theta}!}\bigr)(ab) =
  \bigl(\theta^{c^{}_{\theta}!} \circ \alpha\bigr) (ba)$
  for each\/ $ab \in \cA_{\theta}^{2}$.
\end{itemize}
\end{lem}

\begin{proof}
  By replacing $\theta$ by a power if necessary, we assume that all
  $\theta$-periodic points are fixed, and that, for any letter $a$,
  $\theta(a)$ begins (ends) with a letter $p$ such that $\theta(p)$
  begins (ends) with $p$.  The \emph{interior} of a word
  $w=w^{}_{0} \ts w^{}_{1} \dots w^{}_{n-1} w^{}_{n}$, with $n>1$ to
  avoid pathologies, is defined as the word
  $w^{\circ}=w^{}_{1} \dots w^{}_{n-1}$.  Given any substitution
  $\theta$, the \emph{reverse} substitution $\bar{\theta}$ is defined
  by reflection: If $\theta(a)= p^{}_{1}\dots p^{}_{k}$ for $a\in\cA$,
  then $\bar{\theta} (a) = p^{}_{k}\dots p^{}_{1}$. If $\theta$ is a
  primitive, strongly injective substitution of constant length with
  height\/ $1$ and column number\/ $c^{}_{\theta}$, then so is
  $\bar{\theta}$. Note that
  $ R^{\ts \prime} \! : \, \XX^{}_{\theta} \xrightarrow{\quad}
  \XX_{\bar{\theta}}$
  maps $\theta$-fixed points to $\bar{\theta}$-fixed points.
  \smallskip
 
  Suppose that $\theta$ satisfies the stated conditions, and let
  $G \in \cS(\XX_{\theta})$ with $\kappa(G)=0$. Then, $G$ must send
  $\theta$-fixed points to $\theta$-fixed points, as follows from
  \cite[Prop.~3.19]{CQY} and the remark following Corollary~3.20 in
  the same article.  Now, let $u$, $v$ be distinct fixed points of
  $\theta$, and suppose that $G(u)=v$.  If $G\in \cS(\XX_{\theta})$
  has minimal radius $> 0$, it has left and right radius at most $1$
  by \cite[Prop.~3.19]{CQY}. Suppose (without loss of generality) that
  it has nontrivial right radius. Thus, for some letter $c$, there are
  two occurrences of $c$ in $u$ which are mapped to different letters
  in $v$. In other words, there are letters $d$ and $e$, with
  $d\ne e$, and indices $i\ne j$ such that $u_i = c = u_j$ together
  with $v_i =d$ and $v_j =e$. Due to repetitivity, we may assume that
  $i$ and $j$ are both positive.

  Since $u$ is a fixed point, this implies that, for all $n>0$, the
  word $\theta^n(c)$ appears starting at the indices $r^ni$ and $r^nj$
  in $u$, and then the interior of $G (\theta^n(c))$ appears starting
  at the indices $r^ni +1$ and $r^nj+1$ in $v$.  But, since $v$ is a
  fixed point, $\theta^n(d)$ appears starting at index $r^ni$ in $v$,
  while $\theta^n(e)$ appears starting at index $r^nj$ in $v$. So, for
  all $n$, $(\theta^{n} (d))^\circ = (\theta^{n} (e))^\circ$ and we
  thus have two right-infinite fixed points $u'= d^{\ts\prime} x$ and
  $v'= e' x$ which disagree on their initial entry,
  $d^{\ts\prime}\ne e'$, and which agree on their right rays starting
  at index $1$. This contradicts our assumption that $\theta$ is
  strongly injective.\smallskip
  
  If $H$ is a reversor with $\kappa (H)=0$, it must be of the form
  $H = H' \circ R^{\ts \prime}$ with $H'=H R^{\ts \prime}$ a sliding
  block map of radius at most $1$.  This follows by an application of
  \cite[Prop.~3.25]{CQY} to the sliding block map
  $ H' \! : \, \XX_{\bar{\theta}}\xrightarrow{\quad} \XX^{}_{\theta}$.
  If $\kappa (H)=0$, then $H$ maps $\theta$-fixed to $\theta$-fixed
  points, whence $H'=HR^{\ts \prime} $ maps $\bar{\theta}$-fixed
  points to $\theta$-fixed points.  We can now apply our previous
  argument to $H R^{\ts \prime}$, with the corresponding conclusion.

  Finally, an application of the arguments from the proof of
  \cite[Prop.~3.21]{CQY} gives us the remaining statements, where our
  situation is actually simpler. Since we assume that our automorphism
  (or reversor) has $\kappa$-value $0$, we take $k=0$ in the statement
  of the cited proposition. Also, from above, an automorphism of a
  strongly injective substitution shift has zero radius, and not
  radius $1$, which is what is assumed there.  Thus, conditions (1)
  and (2) of \cite[Prop.~3.21]{CQY} are those that we have to verify,
  the first of which is replaced with our condition (1) above, while
  our commutativity condition (2) is equivalent to (2) in
  \cite[Prop.~3.21]{CQY}. Finally, our statements (3) and (4), which
  apply to reversors, follow after one has initially applied the
  reflection map.
\end{proof}

A minimal shift need not be reversible at all, as we show next. 

\begin{example}\label{ex:non-rev}
Consider the primitive constant-length substitution
\[
    \theta : \quad a \mapsto aba \ts , \, b \mapsto baa \ts ,
\]
which is aperiodic and has a coincidence, so
$\cS (\XX_{\theta}) = \langle S \ts \rangle \simeq \ZZ$ by
Lemma~\ref{lem:Coven}.  Since $aaabaa$ is legal while $aabaaa$ is not,
we have no reflection invariance of $\XX_{\theta}$, so
$R\not\in\Aut (\XX_{\theta})$.  As $a$ is twice as frequent as $b$ in
any $x\in\XX_{\theta}$, which follows from the substitution matrix by
standard Perron--Frobenius theory, $F_{\alpha}$ (defined as in
Example~\ref{ex:flip-rev}) is not in $\Aut (\XX_{\theta})$, and
neither can be $R_{\alpha} = F_{\alpha} R$ or
$R^{\ts\prime}_{\alpha} = F^{}_{\alpha} R^{\ts \prime}$.  Since $\theta$ satisfies the
conditions of the Lemma~\ref{lem:necessary}, any reversor with
$\kappa$-value $0$ must indeed be of the form $R^{\ts\prime}_{\alpha}$. Hence,
there are no reversors with $\kappa$-value zero.

Finally, the only possible $\kappa$-values for a reversor are
integers. To see this, we use Part (2) of
Corollary~\ref{coro:equicontinuous}. To continue, let $\ZZ_n$ denote
the  $n$-adic integers. We now employ the MEF of $\XX_{\theta}$, which
is the ternary odometer $(\ZZ_{3}, +1)$; for a brief but complete
description for our purposes, we refer to \cite[Secs.~3.1 and 3.2 as
well as Thm.~3.6]{CQY}. Note that $\ZZ_{3}$ can canonically be
identified with the ring of $3$-adic integers (hence our notation),
which, in turn, can be used as internal space in a model set
description of $\XX_{\theta}$; compare Example~\ref{ex:pd} below.

Note that $ \{z \in \ZZ_{3} \mid \card(\phi^{-1}(z)) > 1\}$ is the set
of points whose entries (in its tail) eventually all lie in
$\{ 0,1\}$, and it can be verified that
\[
   \{z \in \ZZ_{3} \mid \card (\phi^{-1}(z)) > 1\} \, = \,
    t  - \{z \in \ZZ_{3} \mid \card(\phi^{-1}(z))  > 1\}  
\]
is only possible if $t$ is either eventually $0$, or eventually $2$,
that is, if $t$ is the $3$-adic expansion of an integer. So, by
Corollary~\ref{coro:equicontinuous}, if any reversor were to exist,
there would then also be one with $\kappa$-value $0$, which we have
already excluded. Thus, there are no reversors, and $\cR (\XX)
 = \cS (\XX)$.

Likewise, by completely analogous arguments,
$\theta \! : \, a\mapsto aabba \ts , \, b\mapsto babbb$ has a minimal
symmetry group.  As in the previous example, now with
$\AAA = \ZZ_{5}$, one can verify that no reversor $R^{\ts\prime}_{\alpha}$
satisfies the conditions of Lemma~\ref{lem:necessary}, and one can use
Part (2) of Corollary~\ref{coro:equicontinuous} to check that no
non-integer $\kappa$-value is possible.  \exend
\end{example}

Let us next mention a simple example that is neither
Sturmian nor of constant length.

\begin{example}
  The primitive substitution
\[
   \theta : \quad a \mapsto aab \ts , \; b \mapsto ba \ts ,
\]
emerges from the square of the Fibonacci substitution of
Example~\ref{ex:Fibo} by interchanging the letters in the image of
$b$. It is not Sturmian because
$\cA^{2}_{\theta} = \{ aa, ab, ba, bb \}$, so the complexity is
non-minimal. The shift has pure point spectrum, so is almost
everywhere one-to-one over its MEF, which is an irrational rotation in
this case. Like for the Fibonacci shift from Example~\ref{ex:Fibo},
there is once again a description as a regular model set.  An argument
similar to the one used for the constant-length cases, using
Corollary~\ref{coro:equicontinuous}, shows that the symmetry group is
the minimal one, $\cS (\XX_{\theta}) \simeq \ZZ$.

Moreover, one can also rule out reversibility, hence
$\cR (\XX_{\theta}) = \cS (\XX_{\theta})$ in this case.  This also
shows that the substitution matrix alone, which is the same for this
example and for the square of the Fibonacci substitution, does not
suffice to decide upon reversibility --- the actual order of the
letters matters.  \exend
\end{example}

All examples so far are \emph{deterministic}, hence with zero
topological entropy. Since shifts of finite type generally have large
symmetry groups, it might be tempting to expect the rigidity
phenomenon only for shifts of low complexity. However, it is
well-known \cite{BK} that also shifts with positive entropy can have
minimal symmetry groups. Let us add a recent example that emerges from
a natural number-theoretic setting.

\begin{example}\label{ex:b-free}
  An integer $n$ is called \emph{square-free} if $n$ is not divisible
  by a non-trivial square, so $1,2,3,5,6$ are square-free while
  $4,8,9,12,16,18$ are not. Clearly, $n$ is square-free if and only if
  $-n$ is. The characteristic function of the square-free integers is
  given by $\mu (\lvert n \rvert )^{2}$, where $\mu$ is the M\"{o}bius
  function from elementary number theory, with $\mu (0) := 0$.  Now,
  take the bi-infinite $0$-$1$-sequence
  $x^{}_{\mathrm{sf}} = \bigl(\mu (\lvert n \rvert
  )^{2}\bigr)_{n\in\ZZ}$
  and define a shift $\XX_{\mathrm{sf}}$ via orbit closure. This is
  the \emph{square-free shift} (or flow), which is known to have pure
  point diffraction and dynamical spectrum \cite{BMP,CS}.  In fact,
  $\XX_{\mathrm{sf}}$ is an important example of a dynamical system
  that emerges from a \emph{weak model set} \cite[Sec.~10.4]{TAO}. As
  such, it is a special case within the larger class of $\cB$-free
  shifts; see \cite{Mentzen} and references therein, and
  \cite{Mariusz} for general background.

  Now, it is shown in \cite{Mentzen} that $\XX_{\mathrm{sf}}$ has
  minimal symmetry group, and it is also clear that
  $\XX_{\mathrm{sf}}$ is reflection invariant (because
  $R( x^{}_{\mathrm{sf}} ) = x^{}_{\mathrm{sf}}$), so we are once
  again in the standard situation of Fact~\ref{fact:semi} with
  $\cR (\XX_{\mathrm{sf}}) = \langle S \ts \rangle \rtimes \langle R
  \ts \rangle \simeq \ZZ \rtimes C_{2}$.
  More generally, any $\cB$-free shift $\XX^{}_{\cB}$ is reflection 
  invariant, so we always get $\cR (\XX^{}_{\cB} ) = \cS (\XX^{}_{\cB} )
  \rtimes \langle R \ts \rangle$. We refer to \cite{Mentzen} for the
  general conditions when $\cS (\XX^{}_{\cB}) = \langle
    S \ts \rangle$ holds.  \exend
\end{example}

Let us now proceed with our analysis of some of the classic examples,
including cases with mixed spectrum, hence cases which are multiple
covers of their MEFs.

\begin{example}\label{ex:pd}
  Consider the primitive substitution rule 
\[
    \theta^{}_{\mathrm{pd}} : \quad
     a\mapsto ab \ts , \, b\mapsto aa \ts , 
\]  
which is known as the \emph{period doubling} substitution; see
\cite[Sec.~4.5.1]{TAO} and references given there. It is aperiodic,
has constant length (with height $1$) and a coincidence in the first
position.  So, the corresponding shift $\XX_{\mathrm{pd}}$ is minimal
and has faithful action. Also, it has pure point dynamical spectrum;
compare \cite{Q}. It can be described as a regular model set with
internal space $\ZZ_{2}$, see \cite[Ex.~7.4]{TAO}, and thus possesses
a torus parametrisation with the compact group $\TT=\ZZ_{2} / \ZZ$.

By an application of Fact~\ref{fact:palin}, we know that
$\XX_{\mathrm{pd}}$ is palindromic, and hence reflection
invariant. So, $\cR (\XX_{\mathrm{pd}}) = \cS (\XX_{\mathrm{pd}})
\rtimes \langle R \ts \rangle$ by Fact~\ref{fact:semi}. Moreover, also
in this case, the symmetry group simply is $ \cS (\XX_{\mathrm{pd}}) =
\langle S \ts \rangle \simeq \ts \ZZ$, again as a result of
Lemma~\ref{lem:Coven}.  Consequently, \emph{all} reversors are
involutions.  \exend
\end{example}

\begin{example}\label{ex:TM}
  Related to Ex.~\ref{ex:pd} is the famous Thue--Morse substitution
  \cite{AS,TAO}
\[  
    \theta^{}_{\mathrm{TM}} : \quad
    a\mapsto ab \ts , \, b \mapsto ba \ts .
\] 
This is a primitive, bijective substitution of constant length, hence
we get an extra symmetry of the shift $\XX^{}_{\mathrm{TM}}$, namely
the LEM $F_{\alpha}$ from Example~\ref{ex:flip-rev}. Due to the
aperiodicity of $\XX^{}_{\mathrm{TM}}$, we thus have
\[
    \cS (\XX^{}_{\mathrm{TM}}) \, = \, \langle S \ts \rangle \times
    \langle F_{\alpha} \ts \rangle \, \simeq \, \ZZ \times C_{2}
\]
by Lemma~\ref{lem:Coven}.  Since $\XX^{}_{\mathrm{TM}}$ contains a
palindromic fixed point of $\theta^{\ts 2}_{\mathrm{TM}}$ with seed
$a|a$ (even core case), Fact~\ref{fact:semi} applies and gives
$\cR (\XX^{}_{\mathrm{TM}}) \simeq (\ZZ \times C_{2}) \rtimes C_{2}$,
where the generating involutions, $F_{\alpha}$ and $R^{\ts \prime}$,
commute.  We can thus alternatively write the reversing symmetry group
as
\[
   \cR (\XX^{}_{\mathrm{TM}}) \, \simeq \, ( \ZZ \rtimes
  C_{2}) \times C_{2} \, = \, D_{\infty} \times C_{2} \ts ,
\]
where $D_{\infty}$
denotes the infinite dihedral group. Also in this case, \emph{all}
reversors are involutions; see \cite[Thm.~2 (1)]{BR-Bulletin}.  

Let us also recall that $(\XX^{}_{\mathrm{TM}}, S)$ has mixed spectrum
\cite{Q}. In fact, $\XX^{}_{\mathrm{TM}}$ is a globally $2:1$
extension of $\XX_{\mathrm{pd}}$, compare \cite[Thm.~4.7]{TAO}, and
thus a.e.  $2:1$ over its MEF, which is a binary odometer in this
case, denoted by $(\ZZ_{2}, +1)$. Analogously to our earlier
Example~\ref{ex:non-rev}, $\ZZ_2$ can be viewed as the ring of $2$-adic
integers.  Concretely, the factor map
$\phi \! : \, \XX^{}_{\mathrm{TM}} \xrightarrow{\quad}
\XX_{\mathrm{pd}}$
is defined as a sliding block map acting on pairs of letters, sending
$ab$ and $ba$ to $a$ as well as $aa$ and $bb$ to $b$. This also
implies local derivability in the sense of \cite[Sec.~5.2]{TAO}. The
factor map from $\XX^{}_{\mathrm{TM}}$ to its MEF is obtained as
$\phi$ followed by the usual factor map from $\XX^{}_{\mathrm{pd}}$ to
$\ZZ_{2}$.  \exend
\end{example}

Let us now discuss two ways to generalise these findings to infinite
families of shifts, one over the binary alphabet $\cA$ and another
for larger alphabets.

The \emph{generalised TM} (gTM) substitution \cite{BGG} is a 
substitution in the spirit of \cite{Keane} defined by
\[
    \theta^{(k,\ell)}_{\mathrm{gTM}}  : \quad
   a \mapsto a^k b^\ell \ts , \, b \mapsto b^k a^\ell \ts ,
\]
where $k,\ell \in \NN$ are arbitrary, but fixed. It is primitive,
bijective and of constant length $k+\ell$.  Note that $k=\ell=1$ is
the classic TM substitution from Example~\ref{ex:TM}. For general
$k,\ell$, the corresponding minimal shift
$\XX^{(k,\ell)}_{\mathrm{gTM}}$ is a topological double cover of the
shift $\XX^{(k,\ell)}_{\mathrm{gpd}}$ defined by
\[
    \theta^{(k,\ell)}_{\mathrm{gpd}} : \quad a \mapsto ua \ts ,
    \, b \mapsto ub
\]    
with $u = b^{k-1}a b^{\ell-1}$, the latter being known as the
\emph{generalised period doubling} (gpd) system \cite{BGG}.  In fact,
one may use the factor map $\phi$ from Example~\ref{ex:TM}, and gets
$\XX^{(k,\ell)}_{\mathrm{gpd}} = \phi
\bigl(\XX^{(k,\ell)}_{\mathrm{gTM}} \bigr)$ for all $k,\ell\in\NN$.
Also, we get the same extra symmetry for gTM that we had in the TM
example, namely the LEM $F_{\alpha}$, again by Lemma~\ref{lem:Coven}.
As a generalisation of Example~\ref{ex:pd}, all generalised period
doubling systems can be described as regular model sets, with
$\ZZ_{k+\ell}$ as internal space.

Clearly, for $k=\ell$, the shift $\theta^{(k,\ell)}_{\mathrm{gpd}}$ is
palindromic by Fact~\ref{fact:palin}, so reversible in the simple form
that we know from Fact~\ref{fact:semi}. Since also
$\theta^{(k,\ell)}_{\mathrm{gTM}}$ is palindromic in this case, which
is easy to see from iterating the rule on the legal seed $a|a$, the
situation is analogous to Examples~\ref{ex:TM} and \ref{ex:pd}. If
$k\neq \ell$, one can apply Lemma~\ref{lem:necessary} to show that
$\theta^{(k,\ell)}_{\mathrm{gTM}}$ has no reversors with
$\kappa$-value $0$. All our substitutions
$\theta^{(k,\ell)}_{\mathrm{gTM}}$ have column number $2$, and since
$\{z \in\ZZ_{k+\ell}\mid \card( \phi^{-1}(z) ) > 2 \} = \ZZ$, no
$t \not\in\ZZ$ can satisfy
$ \{z \in\ZZ_{k+\ell}\mid \card(\phi^{-1}(z)) > 2 \} = t -
\{z\in\ZZ_{k+\ell} \mid \card(\phi^{-1}(z)) > 2 \}$.
Hence, by Corollary~\ref{coro:equicontinuous},
$\theta^{(k,\ell)}_{\mathrm{gTM}}$ has no reversors if $k\neq \ell$.

Although $\theta^{(k,\ell)}_{\mathrm{gpd}}$ does not satisfy the
conditions of Lemma~\ref{lem:necessary}, it is still the case that any
reversor with $\kappa$-value $0$ must have radius $0$.  By the
discussion in the proof of Lemma \ref{lem:necessary}, the radius of a
reversor $H$ (meaning that of $H'$ in its representation as
$H = H' \circ R^{\ts \prime}$) is at most $1$, so that a reversor
comprises a local rule that acts on words of length at most $3$.

Now, if one of $k,\ell$ is at least $3$, the only legal
$\theta^{(k,\ell)}_{\mathrm{gpd}}$-words of length $3$ are
$ \{bbb, bba, bab, abb \}$.  By an inspection of the shift defined by
the substitution $\theta^{(k,\ell)}_{\mathrm{gpd}}$, any reversor
(when viewed as a sliding block map with radius $1$) must map
$b\underline{a}b$ to $a$ and each of
$ \{b\underline{b}b, b\underline{b}a, a\underline{b}b \}$ to $b$.
Consequently, the reversor has radius 0.  Now, we check that the
conditions of Lemma~\ref{lem:necessary} are not satisfied, so that
there are no reversors with $\kappa$-value $0$. The situation
$\{k,\ell\}=\{1,2\}$ is similar, except that $bbb$ is not a legal
word. As for the generalised Thue--Morse substitutions, we have
$\{z\in\ZZ_{k+\ell} \mid \card(\phi^{-1}(z)) > 2 \} = \ZZ$, so once
again $\theta^{(k,\ell)}_{\mathrm{gpd}}$ has no reversors if
$k\neq \ell$. Thus we may conclude as follows.

\begin{thm}
  For any\/ $k,\ell\in\NN$, the generalised period doubling shift\/
  $\XX^{(k,\ell)}_{\mathrm{gpd}}$ has minimal symmetry group,
  $\cS_{\mathrm{gpd}} = \langle S \ts \rangle \simeq \ZZ$, while\/
  $\XX^{(k,\ell)}_{\mathrm{gTM}}$ as its topological double cover
  has symmetry group\/
  $\cS_{\mathrm{gTM}} = \langle S \ts \rangle \rtimes \langle
  F_{\alpha} \rangle \simeq \ts \ZZ \times C_{2}$,
  with the involution\/ $F_{\alpha}$ being the LEM for\/
  $a \leftrightarrow b$.
  
  When\/ $k = \ell$, both shifts are reversible, with the standard
  involutory reversor\/ $R^{\ts \prime}$, so one has\/
  $\cR = \cS \rtimes \langle R^{\ts \prime} \ts \rangle \simeq \cS
  \rtimes C_{2}$.  Consequently, all reversors are involutions.

  When\/ $k \ne \ell$, neither substitution is reversible, thus\/
  $\cR = \cS$ in this case.  \qed
\end{thm}

The common feature of the entire family of extensions is that we are
always dealing with a topological $2:1$ mapping between two systems.
Let us now look at different extensions, for which we profit from a
larger alphabet.

\section{Larger alphabets}\label{sec:more}

To discuss another type of extension of the TM shift, let
$C^{}_{\nts N} = \{ 0,1,\ldots , N\! \nts - \! 1\}$ with addition
modulo $N$ denote the cyclic group of order $N$, and define
$\cA = \{ a_{i} \mid i \in C^{}_{\nts N} \}$ as our alphabet. Consider
the substitution
\[
     \theta^{}_{\nts N} : \quad  a_{i} \,\mapsto\, a_{i} \ts 
     a_{i+1} \ts ,  \quad \text{for $i\in C^{}_{\nts N}$} \ts ,
\]
where the sum in the index is understood within $C^{}_{\nts N}$, thus
taken modulo $N$. This defines a primitive substitution, which results
in a periodic sequence (and shift) for $N=1$, and gives the classic TM
system for $N=2$. We denote the shift for $\theta^{}_{\nts N}$ by
$\XX^{}_{\nts N}$, which may be viewed as a cyclic generalisation of the
$N=2$ case.  Note that $N=1$ gives a shift without faithful shift
action, and will later be excluded from the symmetry considerations.
All other cases are aperiodic.

Consider the cyclic permutation
$\pi = (0 \, 1 \, 2 \, \dots \, N\nts\!-\!1)$ of $\cA$, which has
order $N$, and let
$F_{\pi} \! : \, \XX^{}_{\nts N} \xrightarrow{\quad} \XX^{}_{\nts N}$
denote the induced cyclic mapping (or LEM) defined by
$\bigl( F_{\pi}(x)\bigr)_{i} = \pi (x_{i})$. Since
$\theta^{}_{\nts N} \circ \pi = \pi \circ \theta^{}_{\nts N}$ (with
obvious meaning), it is clear that
$F_{\pi} \in \Aut (\XX^{}_{\nts N})$.

\begin{lem}\label{lem:all-legal}
  For any fixed\/ $N\in\NN$, all\/ $N^{2}$ words of length\/ $2$ are
  legal for\/ $\theta^{}_{\nts N}$. Consequently, every\/
  $a_{i} \ts a_{j}$ with\/ $i,j \in C^{}_{\nts N}$ occurs in any
  element of the shift\/ $\XX^{}_{\nts N}$ defined by\/
  $\theta^{}_{\nts N}$.

  Moreover, the map\/ $F_{\pi}$ generates a cyclic subgroup of\/
  $\cS (\XX_{\nts N})$, with\/
  $\langle F_{\pi}\rangle \simeq C^{}_{\nts N}$.
\end{lem}

\begin{proof}
  Since $a_{i-1} \mapsto a_{i-1} \ts a_{i} \mapsto a_{i-1} \ts a_{i}
  \ts a_{i} \ts a_{i+1}$ under the action of $\theta^{}_{\nts N}$, all
  $a_{i} \ts a_{i}$ are legal. As $a_{i} \ts a_{i} \mapsto a_{i} \ts
  a_{i+1} \ts a_{i} \ts a_{i+1}$, also all $a_{i+1} \ts a_{i}$ are
  legal. Iterating this argument inductively shows that all $a_{i+2}
  \ts a_{i}$ are legal, hence also all $a_{i+k} \ts a_{i}$ for any $k
  \in C^{}_{\nts N}$, which proves our first claim.

  The second is a standard consequence of the primitivity of
  $\theta^{}_{\nts N}$, which implies minimality of the shift
  $\XX^{}_{\nts N}$ defined by it as well as the fact that any two
  elements of $\XX^{}_{\nts N}$ are locally indistinguishable, hence
  have the same set of subwords.

  As $F_{\pi}$ obviously commutes with the shift action, the last
  claim is clear.
\end{proof}

Let us now define a sliding block map $\varphi$ via $a_{i} \ts a_{j}
\mapsto a_{i+1-j}$. Since $a_{i+k} \ts a_{j+k} \mapsto a_{i+1-j}$ for
any pair $(i,j)$ and every $k\in C^{}_{\nts N}$, it is not difficult
to check that the factor shift $\XX^{\prime}_{\nts N} := \varphi
(\XX^{}_{\nts N})$ can be generated by the substitution
\[
     \theta^{\ts \prime}_{\nts N} : \quad a_{i} 
     \,\mapsto\, a^{}_{0} \ts a_{i+1} \ts ,
     \quad \text{for $i\in C^{}_{\nts N}$} \ts ,
\]
which is once again primitive. Also, one has a coincidence in the
first position, which gives pure point dynamical spectrum and a
description as a regular model set by the usual arguments.

Observe that $\varphi \circ F_{\pi} = \varphi$, which implies that the
mapping $\varphi$ is at least $N:1$. Now, given an element
$x^{\ts\prime}\in \XX^{\prime}_{\nts N}$, its entry at position $0$,
say, can have at most one of $N$ possibilities as preimage under the
block map, again by Lemma~\ref{lem:all-legal}. Pick one of them, and
continue to the right by one position, where the local preimage is now
uniquely fixed by the previous choice. This argument applies to all
further positions to the right, and analogously also to the left, so
our initial choice results in a unique preimage $x\in\XX^{}_{\nts N}$.
Consequently, $x^{\ts\prime}$ has at most $N$ preimages, and our
previous observation then tells us that $\varphi$ defines a mapping
that is globally $N:1$.

Note that one has $\XX^{\prime}_{1} = \XX^{}_{1}$, which is the
excluded periodic system, while $\XX^{\prime}_{2}$ is the period
doubling system from Example~\ref{ex:pd} with the new alphabet
$\{ a^{}_{0}, a^{}_{1}\}$.  More interestingly, the case $N=3$ is
intimately related with the non-primitive $4$-letter substitution
$\theta^{}_{\mathrm{G}}$ known from the study of the Grigorchuk group;
compare \cite{Vor} and references therein. In fact, using the alphabet
$\{ x,a,b,c \}$, one can define the latter as
\[
      \theta^{}_{\mathrm{G}} : \quad
      x \mapsto xax \ts , \, a \mapsto b \mapsto c \mapsto a \ts ,
\]
which is non-primitive, but has a cyclic structure built into it. The
one-sided fixed point starting with $x$ has an alternating structure
that suggests to use words of length $2$ as new building blocks (or
alphabet), namely $a^{}_{0} = xa$, $a^{}_{1} = xb$ and
$a^{}_{2} = xc$. The resulting substitution rule for the new alphabet
is $\theta^{\ts \prime}_{3}$, which defines a shift that is conjugate
to a higher power (square) shift of the Grigorchuk substitution.  This
observation also provides a simple path to the pure point spectrum of
$\theta^{}_{\mathrm{G}}$ via a suspension into an $\RR$-action.
Alternatively, one can use the method from \cite{Durand} to recode
$ \theta^{}_{\mathrm{G}}$ by a $6$-letter substitution (with
coincidence) that is topologically conjugate to
$\theta^{}_{\mathrm{G}}$ on the level of the $\ZZ$-action generated by
$S$. 

\begin{thm}\label{thm:cyclic-TM}
  Let\/ $N\in\NN$ with\/ $N\geqslant 2$ be fixed. The shift\/
  $\XX^{\prime}_{\nts N}$ has minimal symmetry group\/ $\cS
  (\XX^{\prime}_{\nts N}) = \langle S \ts \rangle \simeq \ts \ZZ$ and is
  reversible, with\/ $\cR (\XX^{\prime}_{\nts N}) \simeq \ts \ZZ \rtimes
  C_{2}$. Consequently, all reversors are involutions.
  
  The cyclic TM shift\/ $\XX^{}_{\nts N}$ is a globally\/ $N:1$ cover
  of\/ $\XX^{\prime}_{\nts N}$, with\/ $\cS (\XX^{}_{\nts N}) \simeq \ts \ZZ
  \times C^{}_{\nts N}$. Moreover, $\XX^{}_{\nts N} $ is reflection
  invariant only for\/ $N=2$, but reversible for all\/ $N\geqslant 2$,
  with an involutory reversor and reversing symmetry group\/ $\cR
  (\XX^{}_{\nts N}) \simeq (\ZZ \times C^{}_{\nts N} ) \rtimes C^{}_{2}$.
  Furthermore, also in this case, all reversors are involutions.
\end{thm}

\begin{proof}
  Since $\theta^{\ts \prime}_{\nts N}$ is of length $2$, the MEF is the
  odometer $(\ZZ_{2},+1)$, with factor map $\phi$. Using the techniques
  of \cite[Sec.~3.4]{CQY}, we see that
\[
   \{ z\in \ZZ_{2} \mid \card 
   (\phi^{-1}(z)) > 1 \} \, = \, \ZZ \ts ,
\]   
where we identify $n\in \ZZ$ with its binary expansion in
$\ZZ_{2}$. Moreover, one has $\kappa (S) = 1$ in this setting. By part
(1) of Corollary~\ref{coro:equicontinuous}, we see that
$\kappa (H) \in \ZZ$ must hold for any
$H\in \cR (\XX^{\prime}_{\nts N})$. Since
$\theta^{\ts \prime}_{\nts N}$ has a coincidence, we have $c=1$, and
part (2) of the same corollary tells us that $\kappa$ is injective on
$\cS (\XX^{\prime}_{\nts N})$, so that the latter is
$\langle S \ts \rangle \simeq \ZZ$ as claimed.
   
Since $\theta^{\ts \prime}_{\nts N}$ is palindromic by
Fact~\ref{fact:palin}, we have $R (\XX^{\prime}_{\nts N}) =
\XX^{\prime}_{\nts N}$ and thus reversibility according to
Fact~\ref{fact:semi} with $C_{2} = \langle R \ts \rangle$ as before,
so the claim on $\cR (\XX^{\prime}_{\nts N})$ is clear.  \smallskip

For $\theta_{\nts N}$, the MEF of $(\XX^{}_{\nts N},S)$ is still
$\ZZ_{2}$, again with
$\{ z \in \ZZ_{2} \mid \card (\phi^{-1}(z)) > 1\} = \ZZ$ and
$\kappa (S) = 1$, so that $\kappa(H) \in \ZZ$ for all
$H\in \cR (\XX^{}_{\nts N})$. The difference to above is that $c=N$
for $\theta_N$, and $\kappa$ is at most $N$-to-$1$ on
$\cS (\XX^{}_{\nts N})$. As
$\kappa^{-1}(0) \cap \cS (\XX^{}_{\nts N})$ contains
$\langle F_\pi \rangle$, and hence equals the latter, one has
$\cS (\XX^{}_{\nts N}) \simeq \ts \ZZ \times C^{}_{\nts N}$.

Unlike in the Thue--Morse case, which is $N=2$, the substitution
$\theta^{}_{\nts N}$ is no longer reflection invariant for $N>2$. To
show that
$\cR (\XX^{}_{\nts N}) \simeq (\ZZ \times C^{}_{\nts N} ) \rtimes
C^{}_{2}$,
it suffices to find one reversor. One can verify that $R^{\ts \prime}$
composed with the LEM defined by the permutation
$\gamma : \, i \leftrightarrow -i$ satisfies the conditions of
Lemma~\ref{lem:necessary}, so that $R_{\gamma}'$ is a reversor (and
then also $R_{\gamma}$). Note that $\gamma \in \varSigma^{}_{\! \cA}$
is an involution ($N>2$) or the identity ($N=2$, and also for $N=1$).
Consequently, $R_{\gamma}$ maps $\XX^{}_{\nts N}$ into itself, and is
an involutory reversor for any $N>1$.  This then shows
$\cR (\XX^{}_{\nts N}) = \cS (\XX^{}_{\nts N}) \rtimes C^{}_{2}$.

For $N=2$, all reversors for $(\XX^{}_{\nts N},S)$ are involutions, as
we already saw in Example~\ref{ex:TM}. This remains true for $N>2$.
In fact, with $\pi = (0\, 1 \, 2 \, \ldots \, N \! - \! 1 )$ from
above, one has $\gamma \circ \pi \circ \gamma = \pi^{-1}$ within
$\varSigma^{}_{\! \cA}$, which implies
\[
     (R^{}_{\gamma} \ts F^{}_{\pi})^{2} \, = \, 
     F_{\pi^{-1} } F^{}_{\pi} \, = \, \id \ts .
\]
Since all reversors are of the form $R_{\gamma} \ts G$ with $G$ a
symmetry, where $G = F^{m}_{\pi} \ts S^{n}$ for some $0\leqslant m <
N$ and some $n\in\ZZ$, a simple calculation now shows the claim.
\end{proof}

So far, most cases of reversibility were based upon the reversor $R$
of Eq.~\eqref{eq:reflect}, or a related involutory reversor. Let us
now look into a classic system that fails to be reflection invariant,
but is nevertheless reversible, while displaying a more complicated
phenomenon than that encountered in Example~\ref{ex:flip-rev} or in
Theorem~\ref{thm:cyclic-TM}.

\begin{example}\label{ex:RS}
  Consider the Rudin--Shapiro substitution, formulated as
\[
   \theta^{}_{\mathrm{RS}}  : \quad
       0 \mapsto 02 \ts , \; 1 \mapsto 32 \ts , \;
       2 \mapsto 01 \ts , \;  3 \mapsto 31 
\]
over the alphabet $\cA = \{ 0,1,2,3\}$, thus following \cite{Q}; see
also \cite[Sec.~4.7.1]{TAO}. This primitive substitution rule commutes
with the homeomorphism of $\{0,1,2,3\}^{\ZZ}$ induced by the
permutation $\beta = (03)(12) \in \varSigma^{}_{\! \cA}$, which thus
also gives a non-trivial involutory symmetry of the RS-shift, denoted
by $F^{}_{\! \beta}$.  Indeed, one has
\[
    \cS (\XX^{}_{\mathrm{RS}}) \, = \, \langle S \ts \rangle \times
    \langle F^{}_{\! \beta} \rangle \, \simeq \, \ZZ \times C^{}_{2} \ts .
\] 
To see this, we use Corollary~\ref{coro:equicontinuous} again. Here,
as in earlier examples, we have $c=2$ and $(\ZZ_{2}, +1)$ as MEF. Like
before, the range of $\kappa$ is contained in $\ZZ$, and $\kappa$ is
at most two-to-one on $\cR (\XX^{}_{\mathrm{RS}})$. The LEM
$F^{}_{\! \beta}$ is then the only additional symmetry, and
$\cS (\XX^{}_{\mathrm{RS}})$ is thus as we claim.

Now, it is known that $\theta^{}_{\mathrm{RS}}$ fails to be
palindromic, compare \cite[Rem.~4.12]{TAO} and references given there,
and by the same method one can check that the hull is not reflection
invariant either, so that $R$ from Eq.~\eqref{eq:reflect} cannot be a
reversor here. However, we may consider a letter permuting reversor
$R_{\alpha}$ as defined in Eq.~\eqref{eq:perm-rev}, with
$\alpha^2 = \beta$. As one can check, $\alpha = (0231)$ as well as
$\alpha^3 = (0132)$ are possible, because $R_{\alpha}$ maps
$\XX^{}_{\mathrm{RS}}$ into itself.  This can be seen as
follows. There are $8$ legal words of length $2$, namely $10$, $20$,
$13$ and $23$ together with their reflected versions.  These four
words are the seeds for the four bi-infinite fixed points under
$\theta^{2}_{\mathrm{RS}}$, each of which can be used to define
$\XX^{}_{\mathrm{RS}}$ via an orbit closure. Now, one has
\[
    \ldots 2|0 \ldots \, 
    \stackrel{R^{\ts\prime}_{\alpha}\,}{\longmapsto} \,
    \ldots 2|3 \ldots \, 
    \stackrel{R^{\ts\prime}_{\alpha}\,}{\longmapsto} \,
    \ldots 1|3 \ldots \, 
    \stackrel{R^{\ts\prime}_{\alpha}\,}{\longmapsto} \,
    \ldots 1|0 \ldots \, 
    \stackrel{R^{\ts\prime}_{\alpha}\,}{\longmapsto} \,
    \ldots 2|0 \ldots
\]
with $R^{\ts\prime}_{\alpha} = S R_{\alpha}$, which then implies that
$\XX^{}_{\mathrm{RS}}$ is invariant under $R^{}_{\alpha}$, where
$R_{\alpha}^{2} = F^{}_{\!\beta}$.

Putting the pieces together, we thus have reversibility with a
reversor of order $4$, and by elementary group theory we get
\[
     \cR (\XX^{}_{\mathrm{RS}}) \, = \, \langle S \ts \rangle \rtimes 
       \langle R_{\alpha} \rangle \, \simeq \, \ZZ \rtimes C_{4} \ts .
\]
Note that there exists no involutory reversor in this case.  In fact,
\emph{all} reversors have order $4$ here, in line with \cite[Thm.~2
(2)]{BR-Bulletin}.  \exend
\end{example}

At this point, it should be clear that there is a lot of freedom for
reversors with other (even) orders. We leave further examples to the
curious reader and turn our attention to higher-dimensional shift
actions.

\section{Extended symmetries of  {$\ZZ^d$}-actions}
\label{sec:extend}

Consider a compact topological space $\XX$ under the continuous action
of $d$ commuting and independent elements
$T_{1} , \ldots , T_{d} \in \Aut (\XX)$, where we again assume
\emph{faithfulness} of the action, which means that
$\cG := \langle T_{1}, \ldots , T_{d} \rangle = \langle T_{1} \rangle
\times \dots \times \langle T_{d} \rangle \simeq \ZZ^{d}$
as a subgroup of $\Aut (\XX)$. This gives a topological dynamical
system $(\XX, \ZZ^{d})$.  Clearly, the prime examples we have in mind
are the classic $\ZZ^{d}$-shift over a finite alphabet $\cA$, so
$\XX = \cA^{\ZZ^{d}}$ and $T_{i} = S_{i}$, and its closed shifts with
faithful $\ZZ^{d}$-action. Here, $S_{i}$ is the shift map that acts in
the $i\ts$th direction, so
\begin{equation}\label{eq:def-d-shift}
     (S^{}_{i} x )^{}_{\bs{n}} \, = \, x^{}_{\bs{n} + \bs{e}_{i}}  \ts ,
\end{equation}
where $\bs{n} \in \ZZ^{d}$ and $\bs{e}_{i}$ denotes the canonical unit
vector for the $i$th coordinate direction.

In any case, with $\ZZ^{d} \simeq \cG \subset \Aut (\XX)$ as above, we
can now define
\begin{equation}\label{eq:def-sym}
    \cS (\XX) \, := \, \mathrm{cent}^{}_{\Aut (\XX)} (\cG) 
\end{equation}
as the \emph{symmetry group} of $(\XX, \ZZ^{d})$. As in the
one-dimensional case, this group will generally be huge (and not
amenable), but under certain rigidity mechanisms it can also be as
small as possible, meaning $\cS (\XX) = \cG$. Next, we define the
\emph{extended symmetry group} as
\begin{equation}\label{eq:def-ext}
     \cR (\XX) \, := \, \mathrm{norm}^{}_{\Aut (\XX)} (\cG) \ts ,
\end{equation}
which gives us back the reversing symmetry group for $d=1$ and
faithful $\ZZ$-action. This is one possible and natural 
way to extend the concept of reversing symmetries to this
multi-dimensional setting.  With the definitions of
Eqs.~\eqref{eq:def-sym} and \eqref{eq:def-ext}, both $\cG$ and $\cS
(\XX)$ are normal subgroups of $\cR (\XX)$, while the detailed
relation between $\cS (\XX)$ and $\cR (\XX)$ needs to be analysed
further.

When the shift action is faithful, so $\cG \simeq \ZZ^{d}$,
more can be said about the structure of $\cR (\XX)$. Note that $h \cG
= \cG h$ is only possible when the conjugation action $g \mapsto h g
h^{-1}$ sends a set of generators of $\cG$ to a (possibly different)
set of generators. Since $\cG$ is a free Abelian group of rank
$d$ by our assumption, its automorphism group is 
$\Aut (\cG) \simeq \GL (d,\ZZ)$, the group of integer 
$d \! \times \! d$-matrices $M$ with $\det (M) \in \{ \pm 1 \}$.
Consequently, there exists a group homomorphism
\begin{equation}\label{eq:psi-def}
   \psi \! : \, \cR (\XX) \xrightarrow{\quad} \GL (d,\ZZ)
\end{equation}
that is induced by the conjugation action. Clearly, the kernel of
$\psi$ is $\cS (\XX)$, thus confirming the latter as a
normal subgroup of $\cR (\XX)$.

Let us take a closer look at the full shift, $\XX_{\cA} =
\cA^{\ZZ^{d}}$.
\begin{lem}\label{lem:full}
  The extended symmetry group of the full shift\/ $\XX_{\cA}$ is given
  by
\[
    \cR (\XX_{\cA}) \, = \, \cS (\XX_{\cA}) \rtimes \cH \ts ,
\]   
where\/ $\cH$ is a subgroup of\/ $\Aut (\XX_{\cA})$ with\/ 
$\cH \simeq \GL (d,\ZZ)$.  In particular, $\cH$ can be chosen
so that, under this isomorphism, each matrix\/ $M \in
\GL (d,\ZZ)$ corresponds to a canonical conjugation with an element
from\/ $\Aut (\XX_{\cA})$.
\end{lem}

\begin{proof}
  Let $M = (m^{}_{ij})^{}_{1\leqslant i,j \leqslant d}$ be a $\GL
  (d,\ZZ)$-matrix and consider the mapping $x \mapsto h^{}_{\nts M}
  (x)$ defined by
\begin{equation}\label{eq:def-hm}
     \bigl( h^{}_{\nts M} (x)\bigr)_{\bs{n}}  \, := 
     \, x^{}_{M^{-1} \bs{n}} \ts ,
\end{equation}
where $\bs{n} = (n^{}_{1}, \ldots , n^{}_{d})^{T}$ is a column vector
and the matrix $M^{-1}$ acts on it as usual. Clearly, $h^{}_{\nts M}$ is a
continuous mapping of $\XX_{\cA}$ into itself and is invertible, hence
$h^{}_{\nts M} \in \Aut (\XX_{\cA})$. One also checks that this
definition leads to $h^{}_{\nts M} h^{}_{\nts M'} = h^{}_{\nts MM'}$,
whence
$\varphi \! : \, \GL (d,\ZZ) \xrightarrow{\quad} \Aut (\XX_{\cA})$
defined by $M\mapsto h^{}_{\nts M}$ is a group homomorphism.  It is
now routine to check that, given $M$, the induced conjugation action
$g \mapsto h^{}_{\nts M} \ts g \ts\ts h^{-1}_{\nts M}$ on $\cG$ sends
$S^{}_{i}$ to $\prod_{j} S^{m^{}_{ji}}_{j}$, for any
$1\leqslant i \leqslant d$, so $\varphi(M) \in \cR(\XX_{\cA})$.  Note
that the iterated conjugation action is consistent with the above
multiplication rule for the $h^{}_{\nts M}$.

Next, consider $\cH := \varphi (\GL (d,\ZZ))$, which is a subgroup of
$\cR (\XX_{\cA})$. Since $\varphi$ is obviously injective, we have
$\cH \simeq \GL (d, \ZZ)$. The group $\cH$ is \emph{canonical} in the
sense that it does not employ any action on the alphabet, but is
constructed only via operations on the coordinates.

Any element $h\in\cR (\XX_{\cA})$ acts, via the above conjugation map,
on the generators of $\cG$, and this induces a mapping $\psi \! : \, \cR
(\XX_{\cA}) \xrightarrow{\quad} \GL (d,\ZZ) \simeq \cH$, with $h
\mapsto \psi (h) = M^{}_{h}$, which is the group homomorphism from
Eq.~\eqref{eq:psi-def}, with kernel $\cS (\XX_{\cA})$.  So,
$\varphi \circ \psi$ is a group endomorphism of $\cR (\XX_{\cA})$ 
with kernel $\cS (\XX_{\cA})$ and image $\GL (d,\ZZ)$. 
Since $\varphi \circ \psi$ acts as the identity on $\cH$ by
construction, we obtain
\[
    \cR (\XX_{\cA}) \, = \, \cS (\XX_{\cA}) \rtimes \cH 
    \, \simeq \, \cS (\XX_{\cA}) \rtimes \GL (d,\ZZ)
\]
as claimed.
\end{proof}

Since $\GL (1,\ZZ)\simeq C_{2}$, the case $d=1$ gives us back our old
result from Fact~\ref{fact:semi}. For $d>1$, the group $\GL (d,\ZZ)$
is infinite, and $\cR (\XX_{\cA})$ is `big'.  Clearly, we will not
always see such a big group for shifts, and many cases will actually
have an extended symmetry group with \emph{finite} index over
$\cS (\XX)$, meaning $[ \cR (\XX) : \cS (\XX)] < \infty$.  In this
case, only a finite subgroup of $\GL (d,\ZZ)$ will be relevant.  This
is connected with the classification of (maximal) \emph{finite}
subgroups of $\GL (d,\ZZ)$, as considered in crystallography; compare
\cite{Schw}. For instance, in the planar case, one could have the
symmetry group of the square, which is isomorphic with $D^{}_{4}$, but
also the one of the regular hexagon, $D^{}_{6}$.  For general $d$, an
interesting subcase emerges when the shift $\XX$ is such that each
element of $\cR (\XX)$ under $\psi$ maps to a signed permutation of
the generators of $\cG$. For reasons that will become clear shortly,
we call such shifts \emph{hypercubic}.

To expand on this, let $\varSigma_{d}$ denote the permutation group of
$\{ 1, \ldots , d \}$ and define $(\pi, \varepsilon)$ with
$\pi\in \varSigma_{d}$ and $\varepsilon \in \{ \pm 1 \}^{d}$ as the
mapping given by
$S^{\pm 1}_{i} \mapsto S_{\pi(i)}^{\pm \varepsilon_{i}}$ for
$1 \leqslant i \leqslant d$. If compared with our previous
description, this means that $(\pi,\varepsilon)$ corresponds to the
matrix $M$ with
$m^{}_{ij} := \varepsilon^{}_{j} \ts \delta_{i, \pi(j)}$, which is the
standard $d$-dimensional matrix representation \cite{MB} of the
hyperoctahedral group $W^{}_{\! d}$.  It is straightforward to check
that we get an induced multiplication rule for the signed
permutations, namely
\begin{equation}\label{eq:wreath}
      (\sigma, \eta) \circ (\pi, \varepsilon) \, = \,
      (\sigma \circ \pi, \eta^{}_{\pi}\ts \varepsilon) \ts ,
\end{equation}
where $(\eta^{}_{\pi})^{}_{i} := \eta^{}_{\pi (i)}$ and the product of
two sign vectors is componentwise (also known as the Hadamard
product).  The inverse is given by $(\pi,\varepsilon)^{-1} =
(\pi^{-1},\varepsilon^{}_{\pi^{-1}})$.  The following property is
well-known; see \cite{MB} and references therein.

\begin{fact}\label{fact:cube}
  For\/ $d\in \NN$, there are\/ $2^{d} \ts d!$ signed
  permutations. Under the multiplication rule of
  Eq.~\eqref{eq:wreath}, they form a group, the wreath product\/
  $W^{}_{\! d} = C^{}_{2} \wr \varSigma_{d} \simeq \varSigma_{d}
  \rtimes C_{2}^{d}$, which is also known as the symmetry group of the
  $d$-dimensional cube or the hyperoctahedral group.  \qed
\end{fact}

Note that $W^{}_{1} \simeq C^{}_{2}$ and $W^{}_{2} \simeq D^{}_{4}$.
The latter is the dihedral group with $8$ elements, which has an
obvious interpretation as the symmetry group of the square. More
generally, the maximal finite subgroups of $\GL (d,\ZZ)$ play a
special role in this setting. Indeed, as a consequence of Selberg's
lemma, one knows that all subgroups of $\GL (d,\ZZ)$ are
virtually torsion-free, which gives the following result.

\begin{fact}\label{fact:torsion}
  A subgroup of\/ $\GL (d,\ZZ)$ is finite if and only if it is
  periodic, i.e., if and only if every element of it is of finite
  order. Consequently, every infinite subgroup of\/ $\GL (d,\ZZ)$
  contains at least one element of infinite order. \qed
\end{fact}

Let us pause to mention an important subtlety here. Though we have
identified a canonical group $\cH$ in Lemma~\ref{lem:full} that is
isomorphic with $\GL (d,\ZZ)$, it is far from unique. In fact, we
could have chosen to include suitable permutations in the alphabet to
augment the group elements, thus forming another group, $\cH'$ say,
that can replace $\cH$ in the statement of the group structure. This
showed up in Example~\ref{ex:flip-rev} and in the proof of
Theorem~\ref{thm:cyclic-TM}. It will also become important for
our later analysis.

For shifts $\XX \subset \XX_{\cA}$, we are now in an analogous (though
more complex) situation to the one-dimensional case. In particular,
for hypercubic shifts say, we have to investigate which elements from
$W^{}_{\! d}$ leave $\XX$ invariant, or, if any element fails in this
respect, whether some combination with a symmetry of $\XX_{\cA}$ steps
in instead.  Here, it may happen that our canonical choice of $\cH$ as
the representative of $W^{}_{\! d}$ fails to capture all extended
symmetries, as we saw in previous examples.

In general, the extended symmetry group need \emph{not} be a semi-direct
product, as is clear from the existence of non-symmorphic space groups
(due to the possibility of glide reflections already for $d=2$).
Still, the most common instance of the extension can be stated as
follows, where
$\psi \! : \, \cR (\XX) \xrightarrow{\quad} \GL (d, \ZZ)$ is the
homomorphism from Eq.~\eqref{eq:psi-def} for faithful shift actions.

\begin{prop}\label{prop:remains}
  Let\/ $\XX \subseteq \cA^{\ZZ^{d}}$ be a closed shift with
  faithful\/ $\ZZ^{d}$-action and symmetry group\/ $\cS (\XX)$.
  Assume further that\/ $\cR (\XX)$ contains a subgroup\/ $\cH$ that
  satisfies\/ $\cH \simeq \psi (\cH)$ together with\/ $\psi (\cH) =
  \psi (\cR (\XX))$. Then, one has the short exact sequence
\[
   1 \, \xrightarrow{\quad} \, \cS (\XX)
   \, \xrightarrow{\, \mathrm{id} \,} \,
   \cR (\XX) \, \xrightarrow{\, \psi \,} \,
   \cH' = \psi (\cH) \, \xrightarrow{\quad} \, 1 \ts ,
\]
and the extended symmetry group of\/ $(\XX, \ZZ^{d})$ is given as
\[
    \cR (\XX) \, = \, \cS (\XX) \rtimes \cH \ts ,
\]   
where the semi-direct product structure is analogous to that of
Lemma~\emph{\ref{lem:full}}.
\end{prop}
   
\begin{proof}
  By assumption, we have $\cG = \ZZ^{d}$, where $\cG$ is the group
  generated by the shift action on $\XX$, and $\cR (\XX) =
  \mathrm{norm}^{}_{\Aut (\XX)} (\cG)$ as in Eq.~\eqref{eq:def-ext}.
  Clearly, $\psi$ is well-defined, with kernel $\cS (\XX)$.

  Let $\cH' = \psi (\cR (\XX))$. By assumption, we have a subgroup
  $\cH$ of $\cR (\XX)$ with $\cH \simeq \cH' = \psi (\cH)$. Via
  composition, we then know that we also have a group homomorphism
  $\psi^{\ts\prime} \! : \, \cR (\XX) \xrightarrow{\quad} \cH$ such
  that $\cS (\XX)$ is the kernel of $\psi^{\ts\prime}$ and that the
  restriction of $\psi^{\ts\prime}$ to $\cH$ is the identity. This
  implies the semi-direct product structure as claimed.
\end{proof}

When the shift from Proposition~\ref{prop:remains} is hypercubic,
$\cH$ is also isomorphic with a subgroup of\/ $W^{}_{\! d}$.  There
are several obvious variants of this statement, and one would like to
know more about the possible group structures. One tool will be the
following generalisation of Theorem~\ref{thm:equicontinuous} to this
setting.

\begin{thm}\label{thm:equi-general}
  Let\/ $(\XX,\ZZ^{d})$ be a shift with faithful shift action and at
  least one dense orbit. Let\/ $\AAA$ be its MEF, with $\phi \! : \,
  \XX \xrightarrow{\quad} \AAA$ the corresponding factor map, where
  the induced\/ $\ZZ^{d}$-action has dense range.

  Then, there is a group homomorphism $\kappa \! : \, \cS (\XX)
  \xrightarrow{\quad} \AAA$ such that
\[
   \phi (G(x)) \, = \, \kappa(G) + \phi (x)
\] 
holds for all $x\in \XX$ and $G \in \cS (\XX)$, so any symmetry\/ $G$
induces a unique mapping\/ $G^{}_{\! \AAA}$ on the Abelian group\/
$\AAA$ that acts as an addition by\/ $\kappa (G)$.

Moreover, if\/ $\kappa( \ZZ^{d})$ is a free Abelian group,
there is an extension of\/ $\kappa$ to a\/ $1$-cocycle of the
action of\/ $\cR (\XX)$ on\/ $\AAA$ such that
\[
    \kappa (GH) \, = \, \kappa (G) + \zeta(G) 
     \bigl( \kappa (H)\bigr)
\]
for all\/ $G,H\in\cR (\XX)$, where\/ $\zeta \! : \, \cR (\XX)
\xrightarrow{\quad} \Aut (\AAA)$ is a group homomorphism.  In
particular, any\/ $G\in\cR (\XX)$ induces a unique mapping on\/ $\AAA$
that acts as\/ $z \mapsto \kappa (G) + \zeta(G) (z)$.
\end{thm}

\begin{proof}
  Let us assume that $\AAA \ne \{ 0 \}$, as the entire statement
  is trivial otherwise.
  The claim on the symmetries follows from the same line of arguments
  employed in the proof of Theorem~\ref{thm:equicontinuous}, where
  we use Eq.~\eqref{eq:def-sym} with $\cG = \ZZ^{d}$. 

  To continue, let us define $a^{}_{i} = \kappa (S_{i})$ for
  $1 \leqslant i \leqslant d$. We know that, for any $z\in\AAA$, the
  set
  $\big\{ z + \sum_{i} n^{}_{i} \ts a^{}_{i} \mid n^{}_{i} \in \ZZ
  \big\}$
  is dense in $\AAA$. In particular, the subgroup $\kappa (\ZZ^{d})$
  is dense in $\AAA$. When $\kappa (\ZZ^d)$ is a free Abelian group,
  one has $\kappa (\ZZ^{d}) \simeq \ZZ^{m}$ for some
  $1 \leqslant m \leqslant d$.

  If $R$ is an extended symmetry, its conjugation action on the shift
  $S_{i}$ is once again given by $S^{}_{i} \mapsto \prod_{j}
  S_{j}^{m^{}_{ji}}$, where $M = \psi (R) \in\GL(d,\ZZ)$ with $\psi$
  from Eq.~\eqref{eq:psi-def}. With
  $S^{\ts \bs{n}} := \prod_{i} S^{\ts n_{i}}_{i}$, one can check that
  $ R \, S^{\ts \bs{n}} R^{-1}  =  S^{\ts M\nts \bs{n}} $.  
   This implies $\phi(R S^{\bs{n}} x) = \phi (R x) + \sum_{i,j}
   a^{}_{i} \ts m ^{}_{ij} \ts n^{}_{j}$, which is to be compared with
   $\phi ( S^{\bs{n}} x) = \phi (x) + \sum_{i} a^{}_{i}\ts n^{}_{i}$.

   Under our assumption,
   $\{ a^{}_{i} \mid 1 \leqslant i \leqslant d \}$ contains a group
   basis of $\kappa (\ZZ^{d}) \simeq \ZZ^{m}$. If $m<d$, some of the
   $a^{}_{j}$ are $0$, without further consequences because
   $\kappa (\ZZ^{d})$ is torsion-free. We can thus set
   $\zeta^{}_{R} (a^{}_{i}) = \sum_{j} a^{}_{j} \, m^{}_{ji}$, which
   defines an element of $\Aut (\kappa(\ZZ^{d}))$ that extends to an
   element of $\Aut (\AAA)$, the latter being zero-dimensional (or
   discrete) in this case.  It is easy to check that
   $G \mapsto \zeta^{}_{G}$ defines a group homomorphism $\zeta$ as
   claimed.

  Now, one can define $\kappa (R) = \phi (Rx) - \zeta^{}_{R} \bigl( \phi
  (x)\bigr)$ and check that this is the required extension. The
  induced action of $R$ on $\AAA$ is then given by $ z \mapsto \kappa
  (R) + \zeta^{}_{R} (z)$ as stated.
\end{proof}

We also need an analogue of Corollary~\ref{coro:equicontinuous},
with focus on the extended symmetries.

\begin{coro}\label{coro:equi-mult}
  Let the setting for a shift\/ $(\XX, \ZZ^d)$ be that of
  Theorem~$\ref{thm:equi-general}$, and define the covering number\/
  $c := \min \{ \card(\phi^{-1}(z)) \mid z \in \AAA \}$.  If\/ $c$ is
  finite, one has that
\begin{enumerate}\itemsep=2pt
\item for each\/ $d\geqslant c$,
\[ 
    \{ z \in \AAA \mid \card(\phi^{-1}(z)) =d \} 
     \, = \, \zeta(G) \ts \bigl(\{z \in
     \AAA \mid \card(\phi^{-1}(z)) = d \}\bigr) + \kappa(G)
\]
holds for each\/ $G\in \cR (\XX)$, and that
\item for each\/ $M\in \GL(d,\ZZ)$, the mapping\/
  $\kappa \! : \, \{ G\in \cR (\XX) \mid \psi(G) = M \}
  \xrightarrow{\quad} \AAA$ is at most\/ $c$-to-one.
  \qed
\end{enumerate} 
\end{coro}

Finally, we need a higher-dimensional analogue of
Proposition~\ref{prop:CHL}.

\begin{prop}\label{prop:higher-CHL}
  Let\/ $(\XX, \ZZ^{d})$ be a shift over a finite alphabet\/ $\cA$,
  with faithful shift action. Then, any\/ $G \in \cS (\XX)$ is
  realised as a $d$-dimensional sliding block map of finite radius,
  which is a local derivation rule in the sense of\/
  \cite[Def.~5.6]{TAO}.
   
  Moreover, any\/ $R\in\cR (\XX)$ is of the form\/ $R = H h^{}_{\nts M}$
  with\/ $M = \psi (R)$ and\/ $h^{}_{\nts M}$ as defined in
  Eq.~\eqref{eq:def-hm}, where\/
  $H \! : \, h^{}_{\nts M} (\XX) \xrightarrow{\quad} \XX$ is again a
  sliding block map of finite radius.
\end{prop}

\begin{proof}
  Since $\cA$ is finite, the claim for $G\in\cS (\XX)$ is nothing but
  the CHL theorem for $\ZZ^{d}$-shifts. Sliding block maps are a
  special case of local derivation rules, which can be considered as
  an extension to the setting of arbitrary patterns of finite local
  complexity in $\RR^{d}$.
  
  Next, consider an element $R\in \cR (\XX)$, define
  $H = R \ts h^{}_{\nts M^{-1}}$ and set $\YY = h^{}_{\nts M} (\XX)$.
  Then, for any $1\geqslant i\geqslant d$, one obtains the diagram
\[
\begin{CD}
   \XX @> h^{}_{\nts M} >> \YY @> H >> \XX \\
   @V  R^{-1} S^{}_{i} R VV   @V T^{}_{i} VV @ VV S^{}_{i} V \\
   \XX @> h^{}_{\nts M} >> \YY @> H >> \XX
\end{CD}
\]
where $T_{i}$ is the shift on $\YY$ defined as in
Eq.~\eqref{eq:def-d-shift}.  Since
$R^{-1} S^{}_{i} R = \prod_{j} S^{m^{-1}_{ji}}_{j}$ via our previously
studied conjugation action, with $\psi(h^{}_{\nts M^{-1}}) = M^{-1}$, one can
easily check that the entire diagram is indeed commutative. In
particular, $H$ intertwines the shift action and must then be a
sliding block map as claimed.
\end{proof}

The map $H$ in Proposition~\ref{prop:higher-CHL} is invertible
by construction, where both $H$ and $H^{-1}$ are represented by
local derivation rules. This means that $\XX$ and $h^{}_{\nts M} (\XX)$
are mutually locally derivable (MLD) from one another; see
\cite[Sec.~5.2]{TAO} for background. This has the following
consequence.

\begin{coro}\label{coro:MLD}
  Under the assumptions of Proposition~\emph{\ref{prop:higher-CHL}}, a
  matrix\/ $M\in\GL (d,\ZZ)$ is an element of\/ $\psi (\cR (\XX))$ if
  and only if\/ $\XX$ and\/ $h^{}_{\nts M} (\XX)$ are MLD.  \qed
\end{coro}

We are now set to consider two classic and paradigmatic examples.

\section{Planar shifts}\label{sec:plane}

Let us begin with an example from tiling theory, namely the chair
inflation (or substitution) tiling; see \cite[Sec.~4.9]{TAO} and
references therein for background. The chair inflation rule, which is
primitive and aperiodic, is illustrated in Figure~\ref{fig:chair}.
From the $D_{4}$-symmetric fixed point depicted there, one defines the
geometric hull as the orbit closure under the translation action of
$\ZZ^{2}$, where we assume that the short edges of the chair prototile
have length $1$.

\begin{figure}[t]
\begin{center}
\includegraphics[width=0.9\textwidth]{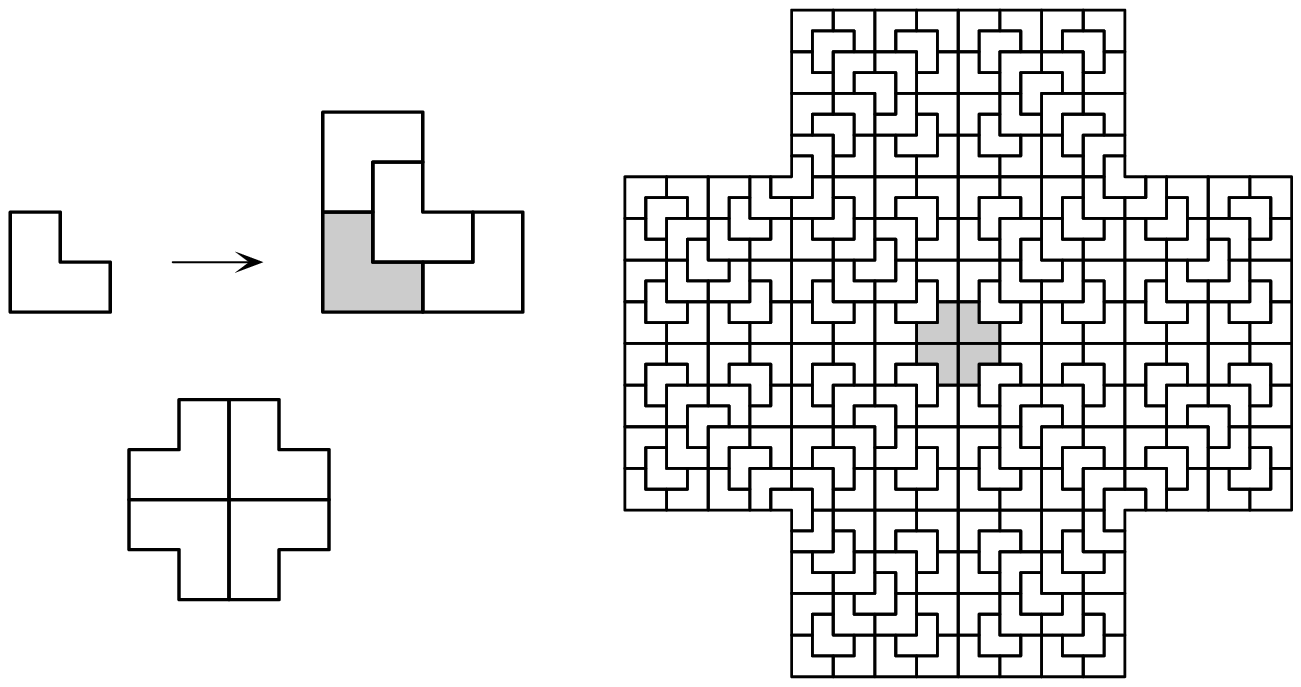}
\end{center}
\caption{The chair inflation rule (upper left panel; rotated tiles are
  inflated to rotated patches), a legal patch with full $D_{4}$
  symmetry (lower left) and a level-$3$ inflation patch generated from
  this legal seed (shaded; right panel). Note that this patch still
  has the full $D_{4}$ point symmetry (with respect to its centre), as
  will the infinite inflation tiling fixed point emerging from
  it.}\label{fig:chair}
\end{figure}

To turn this tiling picture into a symbolic coding such that we can
view it as a traditional $\ZZ^{2}$-shift, we employ the method
introduced in \cite{Robbie}; see also \cite{Olli,TAO}. The coding is
best summarised by an illustration, namely
\begin{equation}\label{eq:code}
    \raisebox{-25pt}{\includegraphics[width=0.6\textwidth]{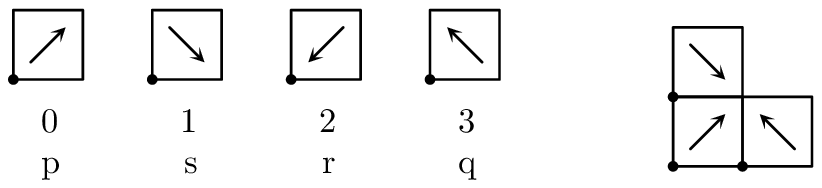}}
\end{equation}
where $p,q,r,s$ are the labels from \cite{Robbie,Olli}, and $0,1,2,3$
those from \cite{TAO}. Using the latter, the central legal patch from
Figure~\ref{fig:chair} is turned into that shown in
Figure~\ref{fig:code}.

By comparison, one clearly sees that the symbolic representation of
the geometric symmetries of the tiling requires the additional use of
permutations. Writing the dihedral group in the usual presentation as
\[
    D_{4} \, = \, \langle \, r,s \mid r^4 = s^2 = (rs)^2 = e \, \rangle 
\]
with $r$ a counterclockwise rotation through $\pi/2$ and $s$ the
reflection in the $y$-axis, its counterpart for the symbolic
description is generated by
\[
      \bigl( r , (0321) \bigr) \quad\text{and}\quad 
      \bigl( s , (03)(12)\bigr) . 
\]
All other pairings follow from here, or can be read off from
Figure~\ref{fig:code}.  In fact, if one wants to use a version of the
group that fixes the configuration of this figure, one needs to add an
application of $S_{1}$ from the left, both for $r$ and for $s$, to
make up for the shifted positions of the symbolic labels relative to
the tile boundaries of the tiling representation. This is reminiscent
of the 1D situation of odd core versus even core palindromes.

\begin{figure}[t]
\begin{center}
   \includegraphics[width=0.3\textwidth]{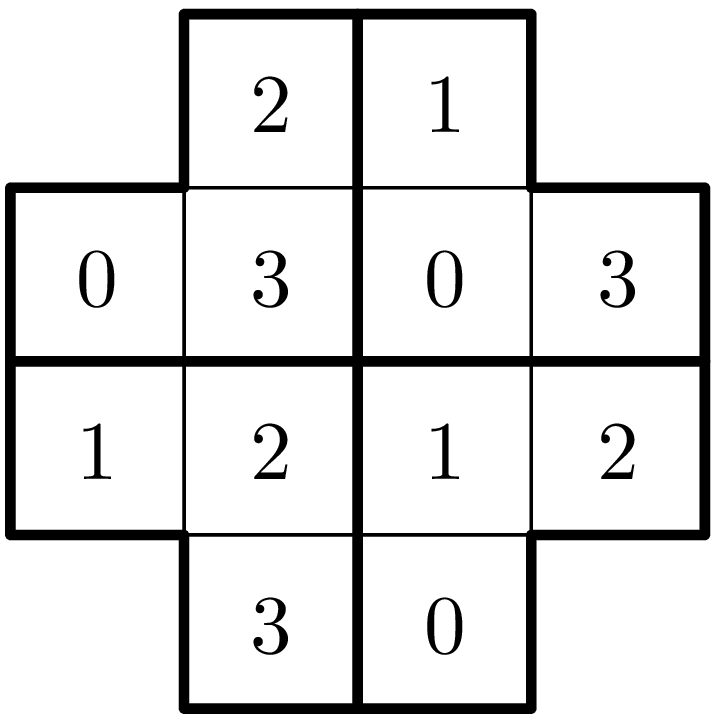}
\end{center}
\caption{The chair tiling seed of Figure~\ref{fig:chair}, now turned
  into a patch of its symbolic representation via the recoding of
  Eq.~\eqref{eq:code}. The relation between the purely geometric point
  symmetries in the tiling picture and the corresponding combinations
  of point symmetries and LEMs can be seen from this
  seed.}\label{fig:code}
\end{figure}

\begin{thm}\label{thm:chair}
  The\/ $\ZZ^{2}$-shift\/ $\ts \XX$ defined by the chair tiling is
  hypercubic and has pure point spectrum. Moreover, it has symmetry
  group $\cS (\XX) = \ts \ZZ^{2}$, which is minimal. The group of
  extended symmetries is\/ $\cR (\XX) = \ts \ZZ^{2} \rtimes \cH$
  with\/ $\cH \simeq W^{}_{2} \simeq D^{}_{4}$, where\/ $D^{}_{4}$
  is a maximal finite subgroup of\/ $\GL (2,\ZZ)$.
\end{thm}

Before we prove this result in the spirit of our earlier arguments
based on the MEF of the system and on its description as a regular
model set, let us use the geometric setting to explain why one should
expect the above claims on the symmetries.  In tiling dynamics, one
usually works with the translation action of $\RR^d$, meaning $\RR^2$
instead of $\ZZ^2$ in the case at hand.  Here, the connection is
established by a suspension with a constant roof function \cite{EW}
on a unit square, which makes the relation rather simple.

Now, observe that any element of $\cR (\XX)$ (in this new setting)
must respect the tiling nature, so it must be a homeomorphism of the
plane that maps tiles onto tiles and, in particular, cannot deform any
of them. This is sufficiently rigid to rule out anything beyond
Euclidean motions. Next, since the geometric symmetries of the fixed
point tiling of Figure~\ref{fig:chair} agree with that of a perfect
square, which is a maximal finite subgroup of $\GL (2,\ZZ)$, we may
conclude that the shift is hypercubic (with $d=2$). 

Let us return to the case of the $\ZZ^2$-action for the proof.  The
claim on the symmetry group is \cite[Thm.~3.1]{Olli}, which is
based on previous work by Robinson \cite{Robbie}, and applies to
both representations of the chair tiling system due to the invariance
of $\cS (\XX)$ under topological conjugacy. We will briefly recall the
argument (in our notation) for convenience.

\begin{proof}[Proof of Theorem~$\ref{thm:chair}$]
  Robinson shows \cite{Robbie} that $\ts \XX$ is an almost everywhere
  one-to-one extension of its MEF, canonically chosen as
  $\AAA = \ZZ_{2}\! \times \! \ZZ_{2}$, so we have $c=1$ for the
  corresponding factor map $\phi$. He also shows that non-singleton
  fibres over $\AAA$ have cardinality either $2$ or $5$, which is not
  difficult to derive from the inflation rule.\footnote{The chair
    tiling can also be viewed as a regular model set with internal
    space $\ZZ_{2} \! \times \! \ZZ_{2}$; compare \cite[Secs.~4.9 and
    7.2]{TAO}. This allows an independent determination of the fibres
    and their cardinality over the MEF. As such, it can be considered
    as a two-dimensional analogue of the period doubling system from
    Example~\ref{ex:pd}.}

  To show that $\ts \XX$ has a trivial symmetry group, we will use the
  $5$-fibres, which are the preimages of $\ZZ^{2} \subset \AAA$ under
  $\phi$.  Here, we may profit from the use of the torus
  parametrisation, with the `torus'
  $\TT = \AAA / \ZZ^2 \simeq (\ZZ_{2}/\ZZ)^2$, which is a compact
  Abelian group. The advantage of it is that we deal with entire shift
  orbits in $\XX$ at once rather than with single elements. Here, any
  $\ZZ^{2}$-orbit of an element of $\XX$ is represented by a point of
  $\TT$, which defines a continuous, a.e.\ one-to-one mapping
  $\varPhi \! : \, \mathrm{orb} (\XX) \xrightarrow{\quad} \TT$, with
  $\varPhi = \chi \circ \phi$, where $\mathrm{orb} (\XX)$ is the space
  of $\ZZ^{2}$-orbits in $\XX$ and
  $\chi \! : \, \AAA \xrightarrow{\quad}\TT$ is the natural
  homomorphism to the factor group $\TT$.  In $\XX$, we have one orbit
  of fibres with cardinality $5$, whose representative on $\TT$ is
  $0$. Invoking Corollary~\ref{coro:equi-mult} for $d=2$, we
  see that $\kappa (\cS (\XX))$ must be in the kernel of $\chi$, which
  means $\kappa (\cS (\XX)) = \ZZ^{2}$.  Since $\kappa$ is injective
  due to $c=1$, the symmetry claim follows.

  Next, $\XX$ is the orbit closure of a configuration (or tiling) with
  full $D_{4}$ symmetry, where the latter is the group discussed
  above, which is generated by the fourfold rotation
  $\left( \begin{smallmatrix} 0 & -1 \\ 1 & 0\end{smallmatrix}\right)$
  and the reflection
  $\left( \begin{smallmatrix} 0 & 1 \\ 1 &
      0 \end{smallmatrix}\right)$.
  This implies $D_{4}$ to be a subgroup of
  $\varPsi := \psi (\cR (\XX)) \subseteq \GL (2,\ZZ)$, and a simple
  calculation confirms that $\cR (\XX)$ indeed contains a subgroup of
  the form $\ZZ^{2} \rtimes \cH$ with $\cH \simeq D_{4}$.  In the
  tiling representation, we can directly work with the linear maps
  defined by the corresponding matrices, while the symbolic
  representation needs the version with additional permutations on the
  alphabet as explained earlier.

  It thus remains to show that $\varPsi$ contains no further element.
  Let us assume the contrary, meaning $D_{4} \subsetneq \varPsi$.
  Since $D_{4}$ is a maximal finite subgroup of $\GL (2,\ZZ)$, we may
  conclude that $\varPsi$ would then be an infinite subgroup, so must
  contain at least one element of infinite order by
  Fact~\ref{fact:torsion}.  The latter cannot be elliptic (which means
  diagonalisable with eigenvalues on the unit circle), so must be
  parabolic or hyperbolic.  Every parabolic element, up to conjugacy
  in $\GL (2,\ZZ)$, is of the form
  $\left( \begin{smallmatrix} 1 & n \\ 0 & 1 \end{smallmatrix}
  \right)$
  for some $0 \ne n \in \ZZ$, while any hyperbolic element has
  irrational eigenvalues, with one expanding and one contracting
  eigenspace, neither of which is a lattice direction. When
  $M\in \GL (2,\ZZ)$ is parabolic or hyperbolic, we have
  $h^{}_{\nts M} (\XX) \ne \XX$, because $h^{}_{\nts M}$ cannot
  preserve (as a set) any pair of orthogonal directions (such as
  $(1,0)$ and $(0,1)$ or the diagonal and the anti-diagonal), and
  hence results in a deformation of the tiling (or a `reshuffling' of
  the symbolic shift) that defines a different hull. In fact, due to
  minimality, one has the stronger property that
  $h^{}_{\nts M} (\XX) \cap \XX = \varnothing $.

  To continue, we employ the $2$-fibres. Any extended symmetry must
  map the set of $2$-fibres onto itself. Each $2$-fibre consists of a
  pair of tilings that agree everywhere except along one line, which is
  either parallel to the main diagonal or to the anti-diagonal. The
  only difference between the two elements in a $2$-fibre is the
  direction of the `stacked' chairs \cite{Robbie}. Now, each
  $\ZZ^{2}$-orbit of $2$-fibres (of which there are uncountably many)
  is parametrised by a point on $\TT$ of the form $(x,x)$ or
  $(x,-x)$. They define two `lines' (or `directions') in $\TT$. This
  parametrising set is $D_{4}$-invariant, as it must be, and $D_{4}$
  is the stabiliser of this set in $\GL (2, \ZZ)$, which we will employ
  shortly. Moreover, this structure remains unchanged under an 
  arbitrary MLD rule, which means
  that both $\TT$ and the directions on $\TT$ stay the same.

  In contrast, our transformations $h^{}_{\nts M}$ will change these
  directional lines, both in the plane (the embedding space of our
  system) and on $\TT$ (the parameter space, where we have an induced
  action of $h^{}_{\nts M}$). If $M$ is parabolic, it can preserve one
  direction (but then not the other, due to the action as a shear) or
  it can map one direction to the other (but then the latter to a new
  one). This second situation can also occur for $M$ hyperbolic. In
  any case, we know that at least one of our two directions is mapped
  to a \emph{new} one.  If we write our hypothetical extended symmetry
  as $H \circ h^{}_{\nts M}$, with $H$ a local derivation rule, we
  thus know that there is a $2$-fibre with two elements that agree
  anywhere off a line $\cL$ and that get mapped by $H$ to two tilings
  which agree off a line $\cL'$ with a direction that is neither
  parallel to the diagonal nor to the anti-diagonal. But this is
  impossible, and $\XX$ and $h^{}_{\nts M} (\XX)$ cannot be MLD. Now,
  Corollary~\ref{coro:MLD} implies $M \notin \varPsi$, and we are done
  via Proposition~\ref{prop:remains}.
\end{proof}

The majority of our previous examples were minimal shifts. It follows
from Example~\ref{ex:b-free} and from \cite{BK,KS} that dynamical
minimality is not necessary for a minimal symmetry group. In fact,
also $\ZZ^{d}$-shifts of algebraic origin quite often display the
necessary rigidity to exclude extra symmetries, irrespective of
complexity or entropy; compare \cite{BS,B-Ward}.

Let $\XX^{}_{\mathrm{L}} \subset \{ 0,1\}^{\ZZ^{2}}$ denote
Ledrappier's shift \cite{Led}, which is defined by the condition that
\begin{equation}\label{eq:def-L}
        x^{}_{(m,n)} + x^{}_{(m+1,n)} + x^{}_{(m,n+1)} \, \equiv \,
        0 \; \bmod 2
\end{equation}
holds for all $x\in\XX^{}_{\mathrm{L}}$ and all $(m,n)\in\ZZ^{2}$; see
Figure~\ref{fig:L} for a geometric illustration of this
condition.  To expand on this, we call a lattice triangle in $\ZZ^{2}$
with vertices $\{ (m,n), (m',n'), (m'',n'')\}$ an $L$-\emph{triangle} if
\[
     x^{}_{(m,n)} + x^{}_{(m',n')} + 
     x^{}_{(m'',n'')} \, \equiv \, 0 \; \bmod 2
\]
holds for all $x \in \XX^{}_{\mathrm{L}}$. By Eq.~\eqref{eq:def-L},
every triangle of the form $\{ (m,n), (m+1,n), (m,n+1)\}$, which we
call \emph{elementary} from now on, is an $L$-triangle.  However,
there are more $L$-triangles. We state the well-known situation as
follows; see \cite[Lemma~5.6]{ABB}.

\begin{fact}\label{fact:L-tri}
  A lattice triangle in\/ $\ZZ^{2}$ is an\/ $L$-triangle if and only
  if it is of the form
\[
       \big\{ (m,n), (m+2^{i},n), (m,n+2^{i}) \big\}
\]    
for some\/ $m,n \in \ZZ$ and\/ $i \in \NN_{0}$. In particular, the
only\/ $L$-triangles with area\/ $\frac{1}{2}$ are the elementary\/
$L$-triangles.  \qed
\end{fact}

%

To rephrase Eq.~\eqref{eq:def-L}, one can say that
$\XX^{}_{\mathrm{L}}$ is the largest subset of the full shift that is
annihilated by the mapping defined by $N:=1 + S^{}_{1} + S^{}_{2}$,
where $x+y$ means pointwise addition modulo $2$. Under this addition
rule, in contrast to most of our previous examples,
$\XX^{}_{\mathrm{L}}$ becomes a compact, zero-dimensional (or
discrete) Abelian group; see \cite{K-book} for background. Here, $N$
is an endomorphism of $\XX^{}_{\mathrm{L}}$, and we have
\[
    \XX^{}_{\mathrm{L}} \, = \, \mathrm{ker} (N) \ts .
\] 
Moreover, $\XX^{}_{\mathrm{L}}$ is not minimal, and has rank-$1$
entropy (but vanishing topological entropy).  Note that
$\XX^{}_{\mathrm{L}}$ is not aperiodic in our above (topological)
sense, though it is still measure-theoretically aperiodic
\cite[Def.~11.1]{TAO}, meaning that all periodic or partially periodic
elements together are still a null set for the Haar measure of
$\XX^{}_{\mathrm{L}}$. In particular, the $\ZZ^{2}$-action on
$\XX^{}_{\mathrm{L}}$ is faithful and mixing \cite{Led}; this can also
be derived from \cite[Thm.~6.5]{K-book}.  The shift action is
\emph{irreducible} in the sense that any closed and shift invariant
true subgroup of $\XX^{}_{\mathrm{L}}$ is finite; compare \cite[Sec.~4
and Ex.~4.3]{KS}. Finally, let us also mention that the spectrum is of
mixed type, with pure point and absolutely continuous components; see
\cite{BW} and references therein.

\begin{figure}[t]
\begin{center}
   \includegraphics[width=0.4\textwidth]{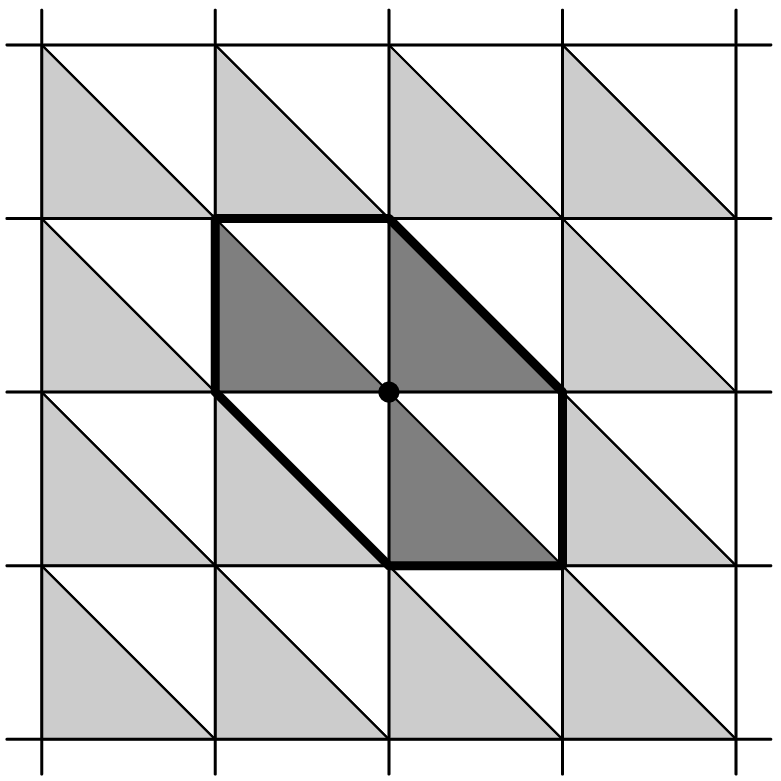}
\end{center}
\caption{Illustration of the central configurational patch for
  Ledrappier's shift condition, which explains the relevance of the
  triangular lattice. Eq.~\eqref{eq:def-L} must be satisfied
  for the three vertices of all elementary $L$-triangles (shaded). 
  The overall pattern of these triangles is preserved by all 
  (extended) symmetries. The group $D^{}_{3}$ from
  Theorem~\ref{thm:L} can now be viewed as the colour-preserving
  symmetry group of the `distorted' hexagon as indicated around the
  origin.  }\label{fig:L}
\end{figure}

\begin{thm}\label{thm:L}
  Ledrappier's shift\/ $\XX^{}_{\mathrm{L}}$ has symmetry group\/ $\cS
  (\XX^{}_{\mathrm{L}}) = \ts \ZZ^{2}$. The group of extended
  symmetries is\/ $\cR (\XX^{}_{\mathrm{L}}) = \ts \ZZ^{2} \rtimes
  D^{}_{3} $, where\/ $D^{}_{3} \simeq C^{}_{3} \rtimes C^{}_{2}$ is
  the dihedral group with\/ $6$ elements.  
\end{thm}

Geometrically, as we shall see, any extended symmetry must map
elementary $L$-triangles to triangles of the same type. The
square lattice representation is thus a `red herring' in this case,
and a more natural representation would use the triangular lattice
instead. This means that the relevant maximal subgroup of
$\GL (2,\ZZ)$ for Ledrappier's shift is $D^{}_{6}$ rather than
$D^{}_{4}$; see Figure~\ref{fig:L} for an illustration.

\begin{rem}
  The cyclic groups $C^{}_3$ and $C^{}_2$
  are generated by the rotation\/
  $G^{}_{3} = \left( \begin{smallmatrix} -1 & -1 \\ 1 &
      0 \end{smallmatrix} \right) $ of order\/ $3$ and
  by the reflection\/ $R^{}_{12} = \left( \begin{smallmatrix} 0 & 1 \\
      1 & 0 \end{smallmatrix} \right)$
  in the space diagonal.  The induced conjugation action of\/
  $G^{}_{3}$ on\/ $\cS (\XX^{}_{\mathrm{L}})$ sends\/ $S^{}_{1}$
  and\/ $S^{}_{2}$ to\/ $S^{-1}_{1} S^{}_{2}$ and\/ $S^{-1}_{1}$,
  respectively.  $R^{}_{12}$ is the involution that interchanges the
  coordinates, and thus conjugates\/ $S^{}_{1}$ into $S^{}_{2}$ and
  vice versa.  \exend
\end{rem}

\begin{proof}[Proof of Theorem~$\ref{thm:L}$]
  {}From Kitchen's result \cite[Observ.~4.6.1]{Kit}, we know that
  every self-conjugacy of $(\XX^{}_{\mathrm{L}},\ZZ^2 )$ must be a
  group isomorphism of $\XX^{}_{\mathrm{L}}$ as an Abelian group (in
  fact, more is true, as shown in \cite[Thm.~1.1 and
  Sec.~4]{KS}). Now, \cite[Thm.~31.1 and Cor.~31.3]{K-book} imply that
  any group isomorphism of $\XX^{}_{\mathrm{L}}$ is a power of a
  shift.  Indeed, any symmetry $G$ has to be affine and has to fix
  $\nix$, the all-zero configuration (which is also the neutral
  element of $\XX^{}_{\mathrm{L}}$ as a group), where $G$ must be a
  group homomorphism.  The only possibility then is that
  $G = S^{\bs{n}} = \prod_{i} S^{n_{i}}_{i}$ for some
  $\bs{n} \in \ZZ^{2}$, as shown in \cite[Sec.~31]{K-book}.  In
  particular, the local condition \eqref{eq:def-L} is \emph{not}
  invariant under the exchange $0 \leftrightarrow \nts\nts 1$ in a
  configuration.

  Now, it is a straightforward to verify that the mapping $h^{}_{\nts M}$
  preserves the pattern of elementary $L$-triangles if and only if
  $M$ is one of the six matrices
\[
   \mbox{\small $
   \left\{   \begin{pmatrix} 1 & 0 \\ 0 & 1
   \end{pmatrix} ,
   \begin{pmatrix} 0 & 1 \\ 1 & 0
   \end{pmatrix} ,
   \begin{pmatrix} -1 & -1  \\ 0  & 1
   \end{pmatrix} ,
   \begin{pmatrix} 0 & 1 \\ -1 & -1
   \end{pmatrix} ,
   \begin{pmatrix} 1 & 0 \\ -1 & -1
   \end{pmatrix} ,
   \begin{pmatrix} -1 & -1 \\ 1 & 0
   \end{pmatrix} \right\} $}.
\]
They form a subgroup of $\GL(2,\ZZ)$ of order $6$, namely
$D_{3} = C_{3} \rtimes C_{2}$; see also Figure~\ref{fig:L}.  For each
such $M$, the mapping $h^{}_{\nts M} = \varphi (M)$ is an element of
$\cR (\XX^{}_{\mathrm{L}})$ with $\psi (h^{}_{\nts M}) = M$, so we see
that the group $\ZZ^{2} \rtimes D_{3}$ is certainly a subgroup of
$\cR (\XX^{}_{\mathrm{L}})$.

Let us next show that $D_{6}$, the unique extension of $D_{3}$ to a
maximal finite subgroup of $\GL (2, \ZZ)$, is \emph{not} contained in
$\psi (\cR (\XX^{}_{\mathrm{L}}))$. Since
$\left( \begin{smallmatrix} 0 & 1 \\ -1 & -1 \end{smallmatrix}
\right)$
has precisely two roots in $\GL (2,\ZZ)$, namely $\pm M$ with
$M= \left( \begin{smallmatrix} 0 & -1 \\ 1 & 1 \end{smallmatrix}
\right)$,
it suffices to exclude $M$. Thus, consider
$\YY := h^{}_{\nts M} (\XX^{}_{\mathrm{L}})$,
which is the shift space defined by
\[
     y^{}_{(m,n)} + y^{}_{(m+1,n)} + 
     y^{}_{(m+1,n-1)} \, \equiv \, 0 \, \bmod{2}
\]
for all $m,n\in\ZZ$, or by
$\YY = \mathrm{ker} (1 + S^{}_{1} + S^{}_{1} S^{-1}_{2})$. The two
shift spaces differ as follows.  In $\XX^{}_{\mathrm{L}}$, when we
know $\{ x^{}_{(k,\ell)} \mid k\in\ZZ \}$ for any (fixed)
$\ell\in\ZZ$, all $x^{}_{(m,\ell+n)}$ with $m\in\ZZ$ and $n\in\NN$ are
uniquely determined.  In contrast, in $\YY$, the knowledge of
$\{ y^{}_{(k,\ell)} \mid k\in\ZZ \}$ determines all
$y^{}_{(m,\ell-n)}$ with $m\in\ZZ$ and $n\in\NN$. In other words, the
knowledge of a configuration along a horizontal line determines
everything above (in $\XX^{}_{\mathrm{L}}$) or below this line (in
$\YY$).

Let us now assume that an invertible local derivation rule
$H \! : \, \YY \xrightarrow{\quad} \XX^{}_{\mathrm{L}}$
exists. Running it along a fixed horizontal line (for $y\in\YY$)
specifies the image (in $\XX^{}_{\mathrm{L}}$) along this line, and
thus (implicitly) also the configuration in the entire half plane
above it. Due to the finite radius ($r$ say) of $H$, this
configuration would be the same for all $y\in\YY$ with the same
configuration in a strip of width $2r$ around the chosen horizontal
line.  Moreover, the values of $x$ below the chosen horizontal line
are effectively determined by the values of $y$ in this strip of
finite width --- irrespective of the values of $y$ at places above the
strip. Since there are certainly distinct elements in $\YY$ that agree
in this strip, we see that $H$ cannot be invertible and a local
derivation rule at the same time, in contradiction to our assumption.

Next, if $\psi (\cR (\XX^{}_{\mathrm{L}}))$ contains any element
beyond $D_{3}$, we must also have at least one element of infinite
order, by an application of Fact~\ref{fact:torsion}.  As in our
previous arguments, such an element must be parabolic or hyperbolic. A
refined version of our above argument will exclude such
elements. Here, one can argue with a \emph{sector} rather than with a
half plane. Corresponding to the three fundamental directions
defined by the sides of the elementary $L$-triangles, there are three
types of index sectors.  If the tip is in the origin, one has
\begin{equation}\label{eq:sector}
      \{ {(m,n)} \mid m,n \geqslant 0 \} \ts , \,
      \{ {(m,-n)} \mid n \geqslant m \geqslant 0 \}  \ts , 
      \, \text{and }
      \{ {(-m,n)} \mid m\geqslant n \geqslant 0 \} \ts ,
\end{equation}
each with the property that the configuration inside the sector is 
completely determined by the configuration on one
of its boundary lines.

Now, for any $M\in\GL(2,\ZZ)$, the shift space 
$ \YY \, = \, h^{}_{\nts M}  (\XX^{}_{\mathrm{L}})$ satisfies 
\[
      \YY \, = \,  \mathrm{ker}
    \bigl(1 + S^{m^{}_{11}}_{1} S^{m^{}_{21}}_{2} + 
    S^{m^{}_{12}}_{1} S^{m^{}_{22}}_{2} \bigr) \ts ,
\]    
as one can derive from the conjugation action on the shift
generators. When $M$ is not of finite order, one has
$\YY \ne \XX^{}_{\mathrm{L}}$, because the elementary triangles for
$\YY$ are then different from the $L$-triangles. Indeed, they are
given by the lattice translates of the triangle with vertices
$\{ (0,0), (m^{}_{11}, m^{}_{21}), (m^{}_{12}, m^{}_{22}) \}$ and thus
are all of the same shape.  Like the elementary $L$-triangles, they have
area $\frac{1}{2}$.  Note that no lattice point other than the three
vertices can lie on any of the sides of an elementary $\YY$-triangle,
and no lattice point can lie in its interior (one way to see this employs
Pick's formula for the area of a lattice triangle).

If one side of an elementary $\YY$-triangle is parallel to a side of
an elementary $L$-triangle, these sides must be congruent because
there cannot be a lattice point in the interior of this side. Now, we
can either repeat our previous argument for a half space (if the two
triangles extend to opposite directions, as above) or we can resort to
one of the three sectors from Eq.~\eqref{eq:sector}. In the latter
case, since the two elementary triangles cannot be congruent, the
selected sector for $\YY$ will differ from that for
$\XX^{}_{\mathrm{L}}$, but they share one boundary line.  If it is
smaller (larger), due to the assumed local nature of the derivation
rule $H$, we find ourselves in the situation that the image
configuration is underdetermined (overdetermined), in contradiction to
the invertibility of $H$. We leave the details of this argument to the
reader.

It remains to consider the case that no side of an elementary
$\YY$-triangle is parallel to that of an elementary $L$-triangle.  If
a $\YY$-sector overlaps with an $\XX^{}_{\mathrm{L}}$-sector, one
boundary line of each sector will (except for $(0,0)$) lie in the
interior of the other, and we may argue once more that our
hypothetical local derivation rule $H$ is overdetermined and thus
cannot be invertible. In the remaining case, the $\YY$-sectors (except
for the point $(0,0)$ again) lie entirely inside the complement of the
sectors from Eq.~\eqref{eq:sector}. In this case, the values of
$y\in\YY$ in the positive quadrant are \emph{not} determined from the
values on $\{ (m,0)\mid m\geqslant 0 \}$, while the values of $x = Hy$
must be.  If $H$ has finite radius and is invertible, this is
impossible.

Put together, this rules out any additional element beyond $D_{3}$ in
$\psi (\cR (\XX^{}_{\mathrm{L}}))$, wherefore we are in the situation
of Proposition~\ref{prop:remains}, which proves our claim.
\end{proof}

It is clear that there are many more examples with small symmetry
groups and interesting extended symmetry groups. Also, for
$d\geqslant 2$, one can have the situation that
$\cS(\XX) \simeq \ZZ^{d}$ is paired with $\cR(\XX) = \cS (\XX)$ or
with $\cR(\XX) \simeq \cS(\XX) \rtimes \GL (d, \ZZ)$, as will be shown
in a forthcoming publication. It remains an interesting open question
to what extent the analysis of extended symmetries can contribute to
some partial structuring of the class of $\ZZ^{d}$-shifts.

\section*{Acknowledgements}

It is a pleasure to thank David Damanik, Franz G\"{a}hler, Dawid
Kielak, Mariusz Lema\'{n}czyk, Klaus Schmidt, Anthony Quas and Luca
Zamboni for useful discussions.  We thank an anonymous reviewer for a
number of insightful suggestions that helped to improve the
manuscript.  This work was supported by the Australian Research
Council (ARC) and also by the German Research Foundation (DFG), within
the CRC 701. MB would like to thank the Erwin Schr\"{o}dinger
Institute (ESI) in Vienna for hospitality, where this manuscript was
completed.

\end{document}